\documentclass[a4paper,12pt,x11names]{article}

\usepackage{csquotes}
\usepackage[USenglish]{babel}
\usepackage{lmodern}

\usepackage{amsmath}
\usepackage{amsfonts}
\usepackage{amssymb}
\usepackage{amsthm}
\usepackage{physics} 
\usepackage{braket}
\usepackage{aligned-overset}
\usepackage{algpseudocode}
\usepackage{algorithm}
\usepackage{easyReview}

\theoremstyle{plain}
\newtheorem{theorem}{Theorem}[section]

\newtheorem{lemma}[theorem]{Lemma}
\newtheorem{proposition}[theorem]{Proposition}

\theoremstyle{definition}

\theoremstyle{remark}
\newtheorem{remark}[theorem]{Remark}

\numberwithin{equation}{section}

\usepackage{xcolor}
\usepackage{mathrsfs} 
\usepackage{enumitem}
\usepackage{boxedminipage2e}
\usepackage[font=small,labelfont=bf]{caption}
\usepackage{subcaption}

\usepackage{mathtools}
\usepackage{nicematrix}
\usepackage{xurl}
\usepackage{todonotes}
\usepackage{abbrev}

\usepackage{titlesec}
\titleformat{\section}{\normalfont\fontsize{14}{16}\bfseries}{\thesection}{1em}{}
\titleformat{\subsection}{\normalfont\fontsize{12}{14}\bfseries}{\thesubsection}{1em}{}
\titleformat{\subsubsection}{\normalfont\fontsize{11}{13}\itshape}{\thesubsubsection}{1em}{}

\usepackage[backend=bibtex, giveninits, style=alphabetic, uniquelist=false, natbib=true, doi=false, url=false, isbn = false,maxnames=8]{biblatex}
\definecolor{BTGreen}{RGB}{0,146,96}
\definecolor{BTDarkGray}{RGB}{72,83,90}
\definecolor{BTMediumGray}{RGB}{127,137,144}
\definecolor{BTLightGray}{RGB}{235,235,228}

\usepackage{hyperref}
\hypersetup{colorlinks=true,urlcolor=BTGreen,linkcolor=BTGreen,citecolor=BTGreen,naturalnames=true,hypertexnames=false}

\usepackage{bib}

\title{Asymptotic equivalence of non-parametric regression with spherical regressors\\and Gaussian white noise}
\author{Martin Kroll\\Universität Bayreuth}
\date{\today}

\allowdisplaybreaks

\begin{document}
\maketitle

\begin{abstract}
We study the asymptotic behavior of both spherical $t$-designs and random uniform designs as the set of sampling points in non-parametric regression with spherical regressors of arbitrary dimension.
We show that the corresponding regression experiments are asymptotically equivalent, in the sense of Le Cam, to the same sequence of Gaussian white noise experiments as the sample size tends to infinity.
More precisely, global asymptotic equivalence is established over spherical Sobolev balls (for both the fixed and the random uniform design case) and over spherical Besov balls (for the fixed design case).
We also derive a matching non-equivalence result showing the sharpness of the imposed smoothness assumptions for any fixed choice of design points.
\end{abstract}


\section{Introduction}

Regression models with spherical regressors of arbitrary dimension and real-valued responses occur in a wide variety of scientific disciplines.
Instances of such models appear in the Earth sciences, where physical processes on the Earth (which is approximately a two-dimensional sphere of radius $6371$ km) 
are considered.
Specific examples are the global temperature field \cite{Oh2004} and the Earth's magnetic field \cite{Holschneider2003}.
Other applications involving functions on the two-dimensional sphere can be found in biology, for instance, for the purpose of cell-shape modeling \cite{Ruan2018}, or in texture analysis \cite{Schaeben2003}.
The special case of functions defined on the three-dimensional sphere arises in crystallography, where it can be used to describe probability distributions of crystalline orientations \cite{Mason2008}, or in medical imaging \cite{Hosseinbor2013}.
Besides, spherical harmonic expansions find various applications in quantum theory \cite{Avery1993}.

Motivated by its wide range of applications, we consider the non-parame\-tric regression model with observations $Z_1,\ldots,Z_n$ obeying model equations
\begin{equation}\label{eq:intro:regression}
	Z_i = f(x_i) + \sigma \epsilon_i, \quad i = 1,\ldots,n,
\end{equation}
where $f\colon \sphere^d \to \R$ is the unknown regression function defined on the $d$-dimensional sphere $\sphere^d = \{ x \in \R^{d+1} : \lVert x \rVert = 1 \}$.
$\XX = \{ x_1,\ldots,x_n\} \subset \sphere^d$ is the finite set of sampling points.
Following the usual notation we denote deterministic sampling points with lowercase letters as in \eqref{eq:intro:regression} and random sampling points with uppercase letters.
$\epsilon_1,\ldots,\epsilon_n$ are i.i.d.\ standard Gaussian random variables independent of the design and the noise level $\sigma > 0$ is assumed to be known.
Non-parametric estimators of the regression function $f$ in model~\eqref{eq:intro:regression} have already been studied:
\cite{Wahba1981} and \cite{Alfeld1996} consider spline interpolation and smoothing on the sphere;
local polynomial smoothing for circular data is treated in \cite{Di_Marzio2009};
series estimators in terms of spherical harmonics and spherical wavelets are studied in \cite{Narcowich2006,Wiaux2008,Monnier2011}.
Let us also mention that, in the context of density estimation with spherical data, kernel methods \cite{Hall1987}, Fourier expansions \cite{Hendriks1990}, and methods based on spherical needlets \cite{Baldi2009a} have already been considered.

The theoretical analysis of the regression model \eqref{eq:intro:regression}, notwithstanding its practical relevance, is hindered by the discrete nature of the model due to the measurements taken at isolated sampling points.
Consequently, one might be tempted to replace model \eqref{eq:intro:regression} by a continuous Gaussian white noise model,
\begin{equation}\label{eq:intro:wn}
	\dd Z(x) = f(x)\dd \mu(x)  + \sigmatilde \dd W(x),
\end{equation}
where $\sigmatilde=\sigmatilde(\sigma,n)>0$ is a suitable noise level, depending on both the noise level $\sigma$ and the sample size $n$ in the discrete model \eqref{eq:intro:regression}, $\mu$ is the normalized surface area measure on $\sphere^d$ and $\dd W$ a standard Gaussian white noise process on $\sphere^d$.
Indeed, model \eqref{eq:intro:wn} allows for a rigorous and neat mathematical analysis, which is conducted in \cite{Klemelae1999Asymptotic} where the sharp asymptotic minimax risk for different function classes and loss functions is derived.
To the best of our knowledge, no theoretical study has yet addressed the relationship between the more realistic observation model \eqref{eq:intro:regression} and the more tractable model \eqref{eq:intro:wn}. The present paper aims to close this gap.

More precisely, the main purpose of the present work is to state (essentially sharp) conditions on the set $\XX$ of sampling points and the class of admissible regression functions that allow replacing the discrete model \eqref{eq:intro:regression} with its continuous surrogate \eqref{eq:intro:wn} (or vice versa).
For this, we rely on Le Cam's theory of asymptotic equivalence of experiments.
Within this theory, the discrepancy between two statistical experiments $\Eexp$ and $\Fexp$ sharing the same parameter space $\Theta$ is quantified using a pseudo-metric $\Delta$, commonly referred to as the Le Cam distance.
Two sequences $(\Eexp_n)$ and $(\Fexp_n)$ of statistical experiments having the same parameter space $\Theta$ are said to be \emph{asymptotically equivalent}, in the sense of Le Cam, if
\begin{equation*}
	\lim_{n \to \infty} \Delta(\Eexp_n,\Fexp_n) = 0.
\end{equation*}
From an inferential point of view, asymptotically equivalent experiments are equally informative in the limit.
For a more comprehensive account of asymptotic equivalence theory, we refer the reader to Chapter~1 of \cite{Gine2016}, as well as the survey papers \cite{Nussbaum2004} and \cite{Mariucci2016}.

Since the seminal paper \cite{Brown1996}, asymptotic equivalence of many non-parametric experiments has been established.
The articles \cite{Brown1996,Rohde2004asymptotic,Reiss2008Asymptotic} consider fixed design regression on the unit interval with equidistant sampling points, whereas \cite{Brown2002} deals with the random design case.
\cite{Reiss2008Asymptotic} also extends the results from \cite{Brown1996} and \cite{Rohde2004asymptotic} to the multivariate and random design case.
In addition, the papers \cite{Brown1996} and \cite{Reiss2008Asymptotic} discuss minor deviations from the assumption of equidistant sampling points in the fixed design setup.
\cite{Grama2002} considers asymptotic equivalence for non-parametric regression with centered, but non-Gaussian, noise, whereas \cite{Meister2013} considers non-regular errors.
The contributions \cite{Carter2009}, \cite{SchmidtHieber2014}, and \cite{Dette2022} weaken the i.i.d.\ assumption on the noise.
The paper \cite{Cai2009} develops asymptotic equivalence theory for robust non-parametric regression.
Limits of asymptotic equivalence theory are discussed in \cite{Brown1998} and \cite{Efromovich1996}.
However, all the papers cited so far discuss asymptotic equivalence for regression experiments and a corresponding Gaussian white noise model only when the regression domain is a subset of some Euclidean space
(admittedly, the case of \emph{periodic} regression functions on $[0,1]^d$, considered also in \cite{Reiss2008Asymptotic}, can be interpreted as a first step to a manifold setup due to the topological identification of $[0,1]^d$ and the $d$-dimensional torus $\mathbb T^d$).

Concerning the choice of the design points, all the asymptotic equivalence results cited above are obtained under the assumption that the sampling points are evenly spread over the whole regression domain. In the deterministic design case, this is achieved by choosing sampling points that form a regular grid; in the random design case, the natural approach is to sample from the uniform distribution on the regression domain.
In the following, we briefly review these two approaches which will be considered in the main part of the paper.

In the one-dimensional case considered in \cite{Brown1996} and \cite{Rohde2004asymptotic}, the target parameter is a function defined on the unit interval, and the canonical deterministic design that yields asymptotic equivalence with the corresponding Gaussian white noise model is the equidistant grid with sampling points $x_i=i/n$ for $i=1,\ldots,n$.
In the multivariate setup, studied in \cite{Reiss2008Asymptotic}, the regression domain is the $d$-dimensional unit cube $[0,1]^d$ and a regular grid of sampling points of the form $(i_1/m,\ldots,i_d/m)$ with $m=n^{1/d}$ and $i_1,\ldots,i_d \in \{ 1,\ldots,m \}$ is assumed.
For more complicated regression domains, possibly subsets of manifolds, a comparable notion of evenly spread deterministic point sets is far from evident.
For the special case of spheres, various measures to evaluate the distributional properties of finite point sets exist \cite{Brauchart2015}.
Besides its intrinsic mathematical motivation, the problem of finding evenly spread points on spheres is of fundamental relevance in fields like viral morphology, crystallography, molecular structure, and electrostatics.
We refer the reader to \cite{Saff1997Distributing} for further discussion.

In the following, we will build on the notion of spherical $t$-designs as originally introduced in \cite{Delsarte1977Spherical}.
A finite set $\XX = \{ x_1,\ldots,x_n \} \subset \sphere^d$ is called a \emph{spherical $t$-design} if the identity
\begin{equation}\label{eq:intro:cubature}
	\frac{1}{n} \sum_{i=1}^n p(x_i) = \int_{\sphere^d} p(x) \dd \mu(x)
\end{equation}
holds for all polynomials $p$ of total degree $\leq t$ in $d+1$ variables.
Since the work of \cite{Seymour1984Averaging}, it is known that spherical $t$-designs exist for all combinations of $t$ and $d$, provided that $n$ is sufficiently large.
In \cite{Bondarenko2013Optimal}, it has finally been proven that spherical $t$-designs in $\sphere^d$ exist for all $n \geq C_dt^d$ with a numerical constant $C_d > 0$ depending only on the dimension of the sphere.
This result is essentially optimal since for the minimal number $N(d,t)$ of points forming a spherical $t$-design, the estimate $N(d,t) \gtrsim t^d$ has already been established in \cite{Delsarte1977Spherical}.

The notion of spherical $t$-designs, originally introduced in the field of algebraic combinatorics, has connections to various fields of mathematics, and we refer to the survey article \cite{Bannai2009} for a comprehensive overview.
In the area of numerical analysis, spherical $t$-designs have already attracted some interest, for instance, as cubature points for numerical integration \cite{Hesse2010}.
Spherical $t$-designs of small cardinality and low dimension are explicitly known and correspond to highly symmetrical point configurations.
For instance, simple examples of spherical $t$-designs on the two-dimensional sphere are given by the vertex sets of the regular tetrahedron (for $t=2$), the cube and the regular octahedron (both for $t=3$), the regular dodecahedron and the regular icosahedron (both for $t=5$).
For larger values of $t$, spherical $t$-designs possessing good geometric properties are also known to exist \cite{Womersley2018}.
Figure~\ref{fig:spherical_designs} illustrates spherical $t$-designs on the two-dimensional sphere for two larger values of $t$, highlighting their excellent distributional properties.

\begin{figure}[t]
	\centering
	\begin{subfigure}[b]{0.45\textwidth}
		\centering
		\includegraphics[width=\textwidth]{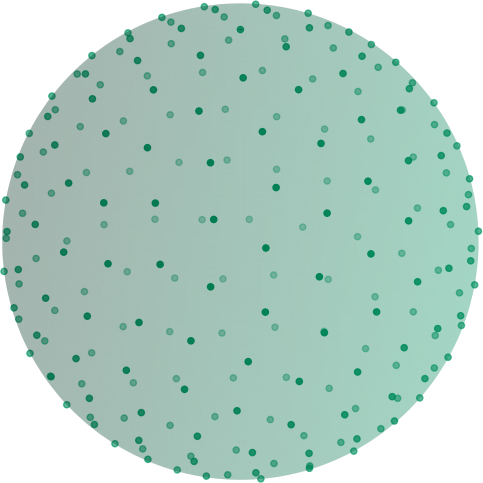}
	\end{subfigure}
	\hfill
	\begin{subfigure}[b]{0.45\textwidth}
		\centering
		\includegraphics[width=\textwidth]{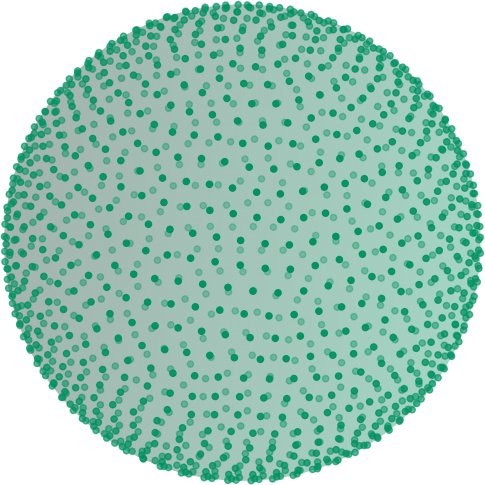}
	\end{subfigure}
	\medskip
	\caption{Examples of spherical $t$-designs for $t=20$ (left plot, $222$ points) and $t=50$ (right plot, 1302 points) on the two-dimensional sphere. The coordinates for both designs are taken from \url{https://web.maths.unsw.edu.au/~rsw/Sphere/EffSphDes/sf.html} (files \texttt{sf020.00222} (left plot) and \texttt{sf050.01302} (right plot)). Darker points are located on the near side of the sphere, lighter points on its far side.}
	\label{fig:spherical_designs}
\end{figure}

In view of these properties, spherical $t$-designs are a promising choice of sampling points in (non-parametric) regression with spherical regressors of arbitrary dimension.
This heuristic is supported in Section~\ref{s:fixed:design} by means of appropriate asymptotic equivalence results.
More precisely, our first main result provides an upper bound on the Le Cam distance between the models \eqref{eq:intro:regression} and \eqref{eq:intro:wn} for a general function class $\Theta$ when the fixed sampling points form a spherical $t$-design.
In the sequel, this general result is applied to two special cases:
First, we prove global asymptotic equivalence for regression functions from spherical Sobolev balls of smoothness $s>d/2$.
Sobolev spaces on spheres can either be defined in terms of charts or via the decay of coefficients in spherical harmonic expansions.
We rely on the latter characterization, which is sufficient for our purposes.
As a second application, which extends the first one, we consider spherical Besov spaces $B^s_{r,q}(\sphere^d)$, which can be defined in terms of the coefficients of a function with respect to a set of spherical needlets. For this more general setup, we derive asymptotic equivalence between the models \eqref{eq:intro:regression} and \eqref{eq:intro:wn} under the assumption $s>d/r$.
This condition and the analogous condition $s>d/2$ in the Sobolev case guarantee that the considered function spaces can be continuously embedded in the Banach space $C(\sphere^d)$ of continuous functions on the sphere.

Proving global asymptotic equivalence over Sobolev and Besov balls builds on results from numerical analysis concerning the approximation by a so-called hyperinterpolation \cite{Sloan1995Polynomial}, which, in statistical terminology, coincides with the least-squares estimator in certain cases.
The hyperinterpolation of the regression function is used to define an intermediate experiment between the experiments defined by \eqref{eq:intro:regression} and \eqref{eq:intro:wn}, respectively.
Since the sample size $n$ must in general be chosen strictly larger than the model dimension of the intermediate experiment, the hyperinterpolation usually does not interpolate the regression function at the design points.
In contrast, in the cited papers for the Euclidean setting, for instance \cite{Reiss2008Asymptotic}, asymptotic equivalence is proven by means of an intermediate experiment defined in terms of a suitable interpolation of the regression function.
Consequently, in the Euclidean case, the intermediate experiment and the regression experiment~\eqref{eq:intro:regression} are even non-asymptotically equivalent.
In the spherical case considered in this work, equivalence in the sense of Le Cam between the intermediate experiment and the regression experiment holds only asymptotically and additional estimates are necessary.
In the general bound on the Le Cam distance between the experiments \eqref{eq:intro:regression} and \eqref{eq:intro:wn}, this leads to an extra term that is not present in the Euclidean setting considered in \cite{Reiss2008Asymptotic}.

In the second part of the paper, we will consider the regression model \eqref{eq:intro:regression} where the design points are i.i.d.\ according to the uniform distribution $\Uc(\sphere^d)$, that is, the normalized surface area measure on the sphere.
This model is more realistic in many applications where sampling points cannot be chosen by the experimenter but are themselves random.
Two instances of random uniform designs, with sample sizes equal to the ones in Figure~\ref{fig:spherical_designs}, are shown in Figure~\ref{fig:random_designs} for the case $d=2$.
\begin{figure}[t]
	\centering
	\begin{subfigure}[b]{0.45\textwidth}
		\centering
		\includegraphics[width=\textwidth]{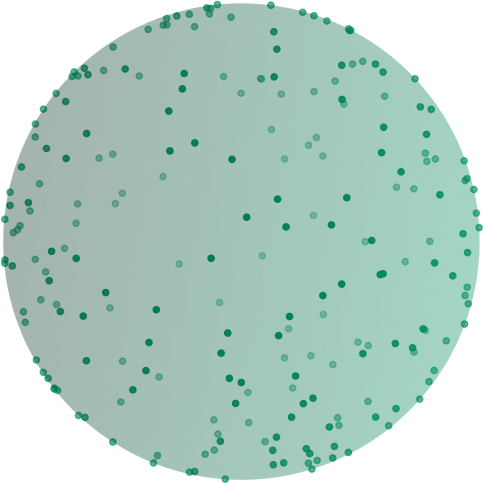}
	\end{subfigure}
	\hfill
	\begin{subfigure}[b]{0.45\textwidth}
		\centering
		\includegraphics[width=\textwidth]{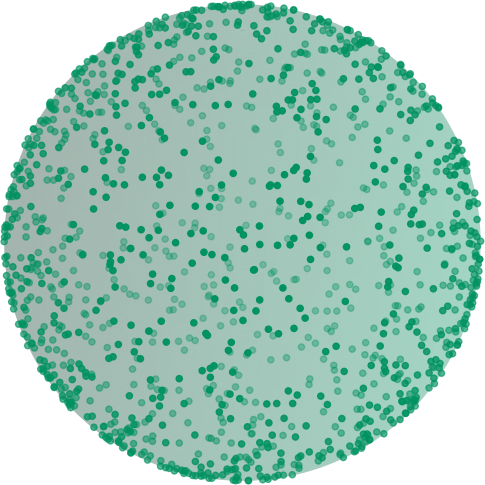}
	\end{subfigure}
	\medskip
	\caption{Examples of random uniform designs of the same size as the spherical $t$-designs in Figure~\ref{fig:spherical_designs} (left plot: 222 points, right plot: 1302 points). Darker points are located on the near side of the sphere, lighter points on its far side.}
	\label{fig:random_designs}
\end{figure}
A comparison of Figures~\ref{fig:spherical_designs} and \ref{fig:random_designs} suggests that a typical realization of a random uniform design is much less regular than a spherical $t$-design with good geometric properties: There exist both data voids and exceptionally close sampling points within the random uniform design.
Furthermore, the exact cubature formula \eqref{eq:intro:cubature} for spherical $t$-designs holds only in expectation in the random uniform design case which makes the analysis much more demanding.
This has already been noticed in \cite{Brown2002} and \cite{Reiss2008Asymptotic} where a separate investigation of low- and high-frequency coefficients of the regression function was necessary in order to establish asymptotic equivalence with a Gaussian white noise model analogous to \eqref{eq:intro:wn}.
This separate treatment of low- and high-frequency coefficients appears also in our analysis, which follows the one from \cite{Reiss2008Asymptotic} in its overall structure. However, some special properties related to the underlying spherical geometry are of importance and additional tools from  numerical analysis and representation theory will be used.
Using these tools asymptotic equivalence of random uniform design regression and the Gaussian white noise model is proven over Sobolev balls on the sphere.
As in the fixed design case, asymptotic equivalence is established for $s > d/2$.

The rest of the paper is organized as follows.
In the preliminary Section~\ref{sec:preliminaries}, we provide some background on spherical harmonic expansions, asymptotic equivalence theory, and the Gaussian white noise model.
In Section~\ref{sec:general_bound} we derive the announced general bound on the Le Cam distance between the regression experiment defined by \eqref{eq:intro:regression} on a spherical $t$-design and the Gaussian white noise model \eqref{eq:intro:wn}.
This bound is used to establish asymptotic equivalence between these models when the target functional parameter belongs to a Sobolev (Section~\ref{sec:AE:Sobolev}) or Besov (Section~\ref{sec:AE:Besov}) ball.
Section~\ref{s:random:design} proves asymptotic equivalence of random uniform design regression and Gaussian white noise over Sobolev balls.
In Section~\ref{sec:non-equivalence}, we prove a matching non-equivalence result showing that the smoothness assumption in Section~\ref{s:fixed:design} cannot be relaxed for any fixed choice of design points.
In Section~\ref{sec:discussion}, we conclude with a brief discussion indicating some connections to optimal design theory and open problems.
All proofs are deferred to Section~\ref{sec:proofs}.

\subsubsection*{Notation}

For sequences $(a_n)$, $(b_n)$ we write $a_n \lesssim b_n$ if there exists a universal (and irrelevant) numerical constant $C > 0$ such that $a_n \leq Cb_n$ holds. The notion $\gtrsim$ is defined analogously and $a_n \asymp b_n$ means that both $a_n \lesssim b_n$ and $b_n \lesssim a_n$ hold simultaneously.
Throughout, we denote the identity operator on some space $S$ by $\Id_S$, and the $m \times m$ identity matrix by $I_m$.
A norm $\lVert \,\boldsymbol \cdot\, \rVert$ without subindex refers to the usual Euclidean norm (where the dimension of the Euclidean space is suppressed in the notation), $\lVert \,\boldsymbol \cdot\, \rVert_2$ denotes the spectral norm. $^\top$ indicates transposition of a matrix or vector.

\section{Preliminaries}\label{sec:preliminaries}

\subsection{Spherical harmonics}\label{subsec:spherical_harmonics}

Let $\mu$ be the normalized surface area measure on the $d$-dimensional sphere $\sphere^d = \{ x \in \R^{d+1} : \lVert x \rVert = 1 \}$ in $d+1$-dimensional ambient Euclidean space.
We denote with
\begin{equation*}
	L^2(\sphere^d) = \left\lbrace f\colon \sphere^d \to \R : \lVert f \rVert^2_{L^2(\sphere^d)} \defeq \int_{\sphere^d} f^2(x) \dd \mu(x) < \infty \right\rbrace
\end{equation*}
the Hilbert space of (equivalence classes of) square-integrable, real-valued functions on $\sphere^d$.
More generally, we consider the Banach spaces
\begin{equation*}
	L^p(\sphere^d) = \left\lbrace f\colon \sphere^d \to \R : \lVert f \rVert^p_{L^p(\sphere^d)} \defeq \int_{\sphere^d} f^p(x) \dd \mu(x) < \infty \right\rbrace, \quad 1 \leq p < \infty,
\end{equation*}
and
\begin{equation*}
	L^\infty(\sphere^d) = \left\lbrace f\colon \sphere^d \to \R : \lVert f \rVert_{L^\infty(\sphere^d)} < \infty \right\rbrace,
\end{equation*}
where $\lVert \,\boldsymbol\cdot\, \rVert_{L^\infty(\sphere^d)}$ denotes the essential supremum.
The Laplace--Beltrami operator $\Delta = \mathrm{div} \circ \nabla$ on $\sphere^d$ gives rise to the decomposition
\begin{equation*}
	L^2(\sphere^d) = \bigoplus_{\ell = 0}^\infty \mathscr H^d_{\ell},
\end{equation*}
where $\mathscr H^d_{\ell}$, $\ell \in \N_0$, denote the eigenspaces associated with the increasing sequence $(\lambda^d_{\ell})$ of non-negative eigenvalues of $-\Delta$.
It is known that $\lambda^d_{\ell} = \ell(\ell+d-1)$ and that the dimension of $\mathscr H^d_{\ell}$ equals
\begin{equation*}
	N_{\ell}^d \defeq \dim \mathscr H^d_{\ell}  = \binom{\ell + d}{d} - \binom{\ell + d -2}{d} \asymp \ell^{d-1}.
\end{equation*}
For any $\ell \in \N_0$, we choose an orthonormal basis $\{ Y_{\ell,m}, m=1,\ldots,N^d_{\ell} \}$ of $\mathscr H^d_{\ell}$ consisting of real-valued eigenfunctions $Y_{\ell,m}$ of $-\Delta$ associated with the eigenvalue $\lambda^d_{\ell}$.
Any such fixed choice of orthonormal basis functions is referred to as spherical harmonics of degree $\ell$.
Any $f \in L^2(\sphere^d)$ can be represented as an infinite series
\begin{equation}\label{eq:fourier:series:spherical:harmonics}
	f = \sum_{\ell=0}^{\infty} \sum_{m=1}^{N^d_{\ell}} \theta_{\ell,m} Y_{\ell,m},
\end{equation}
with coefficients
\begin{equation}\label{eq:def:fourier:coefficients}
	\theta_{\ell,m} = \langle f,Y_{\ell,m} \rangle_{L^2(\sphere^d)} = \int_{\sphere^d} f(x) Y_{\ell,m}(x) \dd \mu(x).
\end{equation}
For convenience, we suppress the dependence of the spherical harmonics $Y_{\ell,m}$ and the corresponding coefficients $\theta_{\ell,m}$ on the dimension $d$ in our notation.
On occasion, it will turn out convenient to switch from the double index notation to a single index. This is achieved by a one-to-one enumeration function
\begin{equation*}
	\iota \colon \{ (\ell,m) \in \N_0 \times \N : 1 \leq m \leq N^d_\ell \} \to \N, \quad (\ell,m) \mapsto \iota(\ell,m),
\end{equation*}
for which we additionally assume that
\begin{equation*}
	\iota(\ell,m) \leq \iota(\ell',m') \quad \Leftrightarrow \quad \ell < \ell' \text{ or } (\ell = \ell' \text{ and } m \leq m').
\end{equation*}
In this case, we set $Y_j = Y_{\ell,m}$ and $\theta_j = \theta_{\ell,m}$ if $j=\iota(\ell,m)$. For instance, the representation \eqref{eq:fourier:series:spherical:harmonics} can alternatively be written as
\begin{equation*}
	f = \sum_{j=1}^\infty \theta_j Y_j.
\end{equation*}
We denote with
\begin{equation*}
	\mathscr P_{L}^d = \bigoplus_{\ell = 0}^L \mathscr H_{\ell}^d
\end{equation*}
the finite-dimensional space spanned by all spherical harmonics up to degree $L \in \N_0$ with dimension
\begin{equation*}
	\dim \mathscr P_L^d = \sum_{\ell=0}^{L} N^d_\ell \asymp L^d.
\end{equation*}
We write $\Pi_V f$ for the $L^2(\sphere^d)$-orthogonal projection of $f$ onto some subspace $V$ of $L^2(\sphere^d)$.
In particular, for $f$ in \eqref{eq:fourier:series:spherical:harmonics} we have
\begin{equation*}
	 \Pi_{\Hscr^d_\ell} f = \sum_{m=1}^{N^d_\ell} \theta_{\ell,m} Y_{\ell,m} \quad \text{and} \quad \Pi_{\Pscr^d_L} f = \sum_{\ell=0}^L \sum_{m=1}^{N^d_\ell} \theta_{\ell,m} Y_{\ell,m}.
\end{equation*}
In the proofs of Section~\ref{sec:proofs}, we will on several occasions exploit the addition formula
\begin{equation}\label{eq:addition:formula}
	\sum_{m=1}^{N^d_\ell} Y_{\ell,m}^2(x) = N^d_\ell, \quad x \in \sphere^d,
\end{equation}
for spherical harmonics (see \cite{Dai2013}, Lemma~1.2.3 and Corollary~1.2.7).

\subsection{Asymptotic equivalence of statistical experiments}\label{sec:recap_asymptotic_equivalence}

We briefly recall the notion of Le Cam distance between statistical experiments and collect some properties that are essential for the remainder of the paper.
Let $\Eexp = (\mathsf X_\Eexp,\mathcal A_\Eexp, (\Pb_{\theta}^\Eexp)_{\theta \in \Theta})$ and $\Fexp = (\mathsf X_\Fexp,\mathcal A_\Fexp, (\Pb^\Fexp_{\theta})_{\theta \in \Theta})$ be two statistical experiments sharing the same parameter space $\Theta$.
Recall that a Markov kernel $\Kf\colon \mathsf X_\Eexp \times \mathcal A_\Fexp \to [0,1]$ induces a map that transports probability measures from $(\mathsf X_\Eexp,\mathcal A_\Eexp)$ to $(\mathsf X_\Fexp,\mathcal A_\Fexp)$.
More precisely, given a probability measure $\Pb_{\theta}^\Eexp$ on $(\mathsf X_\Eexp,\mathcal A_\Eexp)$, the probability measure $\Kf \Pb_{\theta}^\Eexp$ on $(\mathsf X_\Fexp,\mathcal A_\Fexp)$ is defined by
\begin{equation*}
	\Kf \Pb_{\theta}^\Eexp (A) = \int_{\mathsf X_\Eexp} \Kf(x,A) \dd \Pb_{\theta}^\Eexp(x), \quad A \in \mathcal A_\Fexp.
\end{equation*}
The Le Cam distance between $\Eexp$ and $\Fexp$ is defined as the symmetrized quantity
\begin{equation*}
	\Delta(\Eexp,\Fexp) = \max \{ \delta(\Eexp,\Fexp),\delta(\Fexp,\Eexp) \},
\end{equation*}
where the deficiency $\delta(\Eexp,\Fexp)$ is defined by
\begin{equation*}
	\delta(\Eexp,\Fexp) = \inf_{\Kf} \sup_{\theta \in \Theta} V(\Kf\Pb^\Eexp_{\theta}, \Pb^\Fexp_{\theta}).
\end{equation*}
The infimum in this definition is taken over all Markov kernels $\Kf\colon (\mathsf X_\Eexp,\mathcal A_\Eexp) \to (\mathsf X_\Fexp,\mathcal A_\Fexp)$ and $V$ denotes the total variation distance.
If $\Delta(\Eexp,\Fexp) = 0$, the experiments $\Eexp$ and $\Fexp$ are said to be (exactly) equivalent in the sense of Le Cam.
More generally, two sequences $(\Eexp_n)$ and $(\Fexp_n)$ of statistical experiments having the same parameter space $\Theta$ are said to be asymptotically equivalent if
\begin{equation*}
	\lim_{n \to \infty} \Delta(\Eexp_n,\Fexp_n) = 0.
\end{equation*}

\subsection{Gaussian white noise}
Consider the Gaussian white noise model with continuous observation
\begin{equation}\label{eq:app:gaussian:white:noise}
	\dd Z(x) = f(x) \dd \mu(x) + \sigmatilde  \dd W(x),
\end{equation}
where $f \in L^2(\sphere^d)$, $\sigmatilde >0$, and $\dd W$ is standard Gaussian white noise on the $d$-dimensional sphere.
The stochastic differential equation \eqref{eq:app:gaussian:white:noise} can be interpreted in a distributional sense as follows:
\eqref{eq:app:gaussian:white:noise} is equivalent to observing a Gaussian process $G$, indexed by test functions $g \in L^2(\sphere^d)$, defined by
\begin{align*}
	G_g &= \int_{\sphere^d} g(x) \dd Z(x)\\
	&= \int_{\sphere^d} f(x)g(x)  \dd\mu(x) + \widetilde \sigma \int_{\sphere^d} g(x)  \dd W(x).
\end{align*}
Here, the white noise part
\begin{equation*}
	g \mapsto \int_{\sphere^d} g(x) \dd W(x)
\end{equation*}
is a centered Gaussian process with covariance structure
\begin{equation*}
	\Cov \left( \int g(x) \dd W(x), \int h(x) \dd W(x) \right) = \langle g,h \rangle_{L^2(\sphere^d)}.
\end{equation*}
Evaluating the process $G$ at an orthonormal basis $\{ \varphi_j \}_{j \in \N}$ of $L^2(\sphere^d)$ shows that observing \eqref{eq:app:gaussian:white:noise} is equivalent to the infinite-dimensional Gaussian sequence model
\begin{equation*}
	y_j = G_{\varphi_j} = \theta_j + \widetilde \sigma \eta_j, \quad j\in \N,
\end{equation*}
where $\theta=(\theta_j)$ with $\theta_j = \langle f,\varphi_j \rangle_{L^2(\sphere^d)}$ is the sequence of Fourier coefficients (with respect to the basis $\{ \varphi_j \}_{j \in \N}$) and $(\eta_j)$ is an i.i.d.\ sequence of standard Gaussian random variables.
Let $\Pb_{f}$ denote the distribution of the process $Z$ in \eqref{eq:app:gaussian:white:noise}.
The following expression for the total variation distance between $\Pb_f$ and $\Pb_{g}$ is derived in \cite{Carter2006continuous}, Section~3.2, and will be used frequently in the proofs of Section~\ref{sec:proofs}:
\begin{equation}\label{eq:app:bound:tv:gwnm}
	V(\Pb_f,\Pb_{g}) = 1 - 2\Phi \left( - \frac{\lVert f-g \rVert_{L^2(\sphere^d)}}{2\sigmatilde} \right) \lesssim \sigmatilde^{-1} \, \lVert f-g \rVert_{L^2(\sphere^d)},
\end{equation}
where $\Phi$ denotes the distribution function of a standard Gaussian random variable.

\section{Regression on spherical $t$-designs}\label{s:fixed:design}

For a class $\Theta$ of functions $f\colon \sphere^d \to \R$ and a finite set $\XX = \{ x_1,\ldots,x_n \} \subset \sphere^d$ of fixed sampling points,
we denote by $\Fexp^d_{n}=\Fexp^d_{n}(\Theta,\XX)$ the regression experiment with observations
\begin{equation}\label{eq:general:regression}
	Z_i = f(x_i) + \sigma \epsilon_i, \quad i=1,\ldots,n,
\end{equation}
where $f\in \Theta$ is the unknown regression function, $\epsilon_1,\ldots,\epsilon_n$ are i.i.d.\ standard Gaussian, and $\sigma > 0$.
We assume that the standard deviation $\sigma$ of the additive noise is known and suppress the dependence of the experiment $\Fexp^d_{n}$ on $\sigma$ in the notation.
Similarly, we denote by $\Gexp^d_{n}=\Gexp^d_{n}(\Theta)$ the Gaussian white noise experiment defined via Equation~\eqref{eq:intro:wn} in the introduction with $\widetilde \sigma = \sigma/\sqrt n$, that is, the experiment given by
\begin{equation}\label{eq:general:wn}
	\dd Z(x) = f(x)\dd \mu(x)  + \frac{\sigma}{\sqrt n} \dd W(x).
\end{equation}

\subsection{A general bound for the Le Cam distance}\label{sec:general_bound}

In order to state a general bound on the Le Cam distance $\Delta(\Fexp^d_{n},\Gexp^d_{n})$, we introduce some further notation.
Given a finite set $\{ \psi_1,\ldots,\psi_m \}$ of approximating functions in $L^2(\sphere^d)$, not necessarily forming an orthonormal basis of their span $S = \linspan(\{ \psi_1,\ldots,\psi_m \})$, we consider the finite-dimensional approximation
\begin{equation}\label{eq:truncated:approximation}
	f_m = \sum_{j=1}^{m} \beta_j \psi_j
\end{equation}
of a function $f\in L^2(\sphere^d)$ where
\begin{equation*}
	\beta_j = \langle f,\psi_j \rangle_{L^2(\sphere^d)} = \int_{\sphere^d} f(x)\psi_j(x) \dd \mu(x).
\end{equation*}
The empirical counterparts $\betatilde_j$ of the coefficients $\beta_j$ are obtained by the corresponding equal-weight cubature rules at the sampling points $\XX=\{ x_1,\ldots,x_n \}$,
\begin{equation}\label{eq:def:betatilde:cubature}
	\betatilde_j = \langle f,\psi_j \rangle_n = \frac{1}{n} \sum_{i=1}^n f(x_i)\psi_j(x_i).
\end{equation}
Since spherical harmonics are restrictions of polynomials in $d+1$ variables to the unit sphere with the index $\ell$ indicating the degree of the polynomial associated with $Y_{\ell,m}$ (see \cite{Dai2013}, Chapter~1, for details), a spherical $t$-design $\XX=\{ x_1,\ldots,x_n \}$ with $t=L$ yields an equal-weight cubature rule that is exact for functions in $\mathscr P_{L}^d$, that is,
	\begin{equation*}
		\frac{1}{n} \sum_{i=1}^{n} f(x_i) = \int_{\sphere^d} f(x) \dd \mu(x), \quad f \in \mathscr P^d_{L}.
\end{equation*}

The semi-norm associated with the inner product $\langle \boldsymbol \cdot,\boldsymbol{\cdot} \rangle_n$ is denoted by $\lVert \,\boldsymbol{\cdot}\, \rVert_n$.
Replacing $\beta_j$ in \eqref{eq:truncated:approximation} with $\betatilde_j$ leads to the empirical approximation
\begin{equation}\label{eq:def:ftilde_m}
	\ftilde_{m} = \sum_{j=1}^m \betatilde_j \psi_j.
\end{equation}
Note that while $f_m$ depends on $f$ via exact $L^2(\sphere^d)$-inner products, the approximation $\ftilde_m$ is fully data-driven by means of the cubature formula \eqref{eq:def:betatilde:cubature}.
Associated with $\ftilde_m$ we introduce the intermediate experiment $ \Gtildeexp^d_{n}= \Gtildeexp^d_{n}(\Theta,\XX)$ with continuous observation
\begin{equation}\label{eq:intermediate}
	\dd \widetilde Z(x) = \ftilde_m(x) \dd \mu(x) + \frac{\sigma}{\sqrt n} \dd W(x).
\end{equation}
Note that the intermediate experiment $\Gtildeexp^d_{n}$, in contrast to $\Gexp_{n}^d$, depends on the sampling points $\XX$ via the empirical coefficients $\betatilde_j$.
The following theorem states the announced upper bound on the Le Cam distance $\Delta(\Fexp_{n}^d,\Gexp_{n}^d)$.

\begin{theorem}\label{thm:fixed:general}
	Consider the experiments $\Fexp_{n}^d$, defined by \eqref{eq:general:regression}, and $\Gexp_{n}^d$, defined by \eqref{eq:general:wn}, where the unknown parameter $f$ belongs to some class $\Theta$ of functions defined on the $d$-dimensional sphere~$\sphere^d$. 
	Assume that $S = \linspan(\{ \psi_1,\ldots,\psi_m \}) \subseteq \mathscr P^d_{L}$ for some $L \in \N_0$ and that $\XX = \{ x_1,\ldots,x_n \} \subset \sphere^d$ is a spherical $t$-design for $t \geq 2L$.
	Then, the Le Cam distance between the two experiments is bounded by
	\begin{align*}
		\Delta(\Fexp^d_{n},\Gexp^d_{n}) \leq \Delta_1 + \Delta_2,
	\end{align*}
	where
	\begin{align*}
		\Delta_1 = \Delta(\Fexp^d_{n},\widetilde \Gexp^d_{n})&\leq  1 - 2\Phi \left( -\frac{\sqrt n}{2\sigma} \sup_{f \in \Theta} \, \lVert f- \ftilde_m \rVert_n \right)\\
		&\lesssim \sigma^{-1}n^{1/2} \sup_{f \in \Theta} \, \lVert f- \ftilde_m \rVert_n
	\end{align*}
	and
	\begin{align*}
		\Delta_2 = \Delta(\widetilde \Gexp^d_{n},\Gexp^d_{n}) &\leq 1 - 2\Phi  \left( - \frac{\sqrt n}{2\sigma}\, \sup_{f \in \Theta} \, \lVert f - \ftilde_m \rVert_{L^2(\sphere^d)} \right)\\
		&\lesssim \sigma^{-1} n^{1/2} \sup_{f \in \Theta}\, \lVert f - \ftilde_m \rVert_{L^2(\sphere^d)}.
	\end{align*}
\end{theorem}

\begin{remark}
	Let us briefly comment on the two terms appearing in the bound on $\Delta(\Fexp_n^d,\Gexp_n^d)$ derived in the theorem.
	The second term,
	\begin{equation*}
		\Delta_2 \lesssim \sigma^{-1}n^{1/2} \sup_{f \in \Theta}\, \lVert f - \ftilde_m \rVert_{L^2(\sphere^d)},
	\end{equation*}
	is not surprising as it already appears in known results for the multivariate Euclidean case (see, for instance, the bound stated in Theorem~2.4 of \cite{Reiss2008Asymptotic}).
	In contrast, the first term,
	\begin{equation*}
		\Delta_1 \lesssim \sigma^{-1}n^{1/2} \sup_{f \in \Theta} \, \lVert f- \ftilde_m \rVert_n,
	\end{equation*}
	is new.
	In the setting of Theorem~\ref{thm:fixed:general}, the approximant $\ftilde_m$ cannot in general be chosen to interpolate $f$ at the sampling points $x_1,\ldots,x_n$.
	This is due to the non-existence, except in a few special cases, of so-called \emph{tight} spherical $t$-designs.
	A more detailed discussion of this issue will be given in Remark~\ref{rem:tight_t-designs}.
\end{remark}

\begin{remark}
	The proof of Theorem~\ref{thm:fixed:general} is constructive in the sense that it provides explicit Markov kernels for transferring observations between the two considered experiments.
	These kernels involve randomizations, which, however, require knowledge of the noise level $\sigma$.
	This assumption is common in the literature on asymptotic equivalence between non-parametric regression and Gaussian white noise experiments (see, for instance, \cite{Brown2002,Rohde2004asymptotic,Reiss2008Asymptotic}).
	Only a few works address the more realistic case with an unknown noise variance; \cite{Carter2007} is a notable contribution in this direction.
\end{remark}

\begin{remark}
	Theorem~\ref{thm:fixed:general} implies global asymptotic equivalence if the class $\Theta$ consists of functions $f$ that satisfy the smoothness condition
	\begin{equation}\label{eq:cond:themistoclakis}
		\lim_{n \to \infty} n^{1/2} \sup_{f \in \Theta} \inf_{p \in \mathscr P^d_L} \, \lVert f - p \rVert_{L^\infty(\sphere^d)} = 0
	\end{equation}
	with $L=L(n) \to \infty$ as $n \to \infty$.
	This smoothness assumption appears in the numerical analysis literature since it guarantees uniform convergence of the hyperinterpolation $\ftilde_m$ to $f$ (see \cite{Themistoclakis2019}, p.~1091).
	For instance, for band-limited functions, \eqref{eq:cond:themistoclakis} is trivially fulfilled.
	More precisely, for the function class
	\begin{equation*}
		\Theta_{L^\ast} = \left\lbrace  f\colon \sphere^d \to \R : f = \sum_{\ell=0}^{L^\ast} \sum_{m=1}^{N^d_\ell} \theta_{\ell,m} Y_{\ell,m} \right\rbrace
	\end{equation*}
	for some \emph{fixed} $L^\ast \in \N_0$, both the considered regression model $\Fexp_n^d$ and the Gaussian white noise model $\Gexp_n^d$ are exactly equivalent for sufficiently large $n$ to observing a multivariate Gaussian of dimension $\dim \mathscr P^d_{L^\ast} =\sum_{\ell=0}^{L^\ast} N^d_\ell$ with unknown mean in $\R^{\dim \mathscr P^d_{L^\ast}}$ and covariance matrix $\sigma^2 n^{-1} I_{\dim \mathscr P^d_{L^\ast}}$.
\end{remark}

\subsection{Results for specific function classes}

In the following two sections, we apply the general Theorem~\ref{thm:fixed:general} to two specific function classes: Sobolev balls (Section ~\ref{sec:AE:Sobolev}) and Besov balls (Section~\ref{sec:AE:Besov}).

\subsubsection{Asymptotic equivalence over Sobolev balls}\label{sec:AE:Sobolev}

For a smoothness parameter $s>0$ (not necessarily an integer), we define the (fractional) Sobolev norm $\lVert \,\boldsymbol \cdot\, \rVert_{H_2^s(\sphere^d)}$ by
\begin{equation*}
	\lVert f \rVert_{H_2^s(\sphere^d)}^2 = \sum_{\ell=0}^{\infty} (1+\ell(\ell+d-1))^s \sum_{m=1}^{N^d_{\ell}} \theta_{\ell,m}^2, \quad f \in L^2(\sphere^d),
\end{equation*}
where $(\theta_{\ell,m})$ is the sequence of Fourier coefficients of the function $f$ defined via \eqref{eq:def:fourier:coefficients}.
The Sobolev space $H_2^s(\sphere^d)$ is defined as the set of functions $f \in L^2(\sphere^d)$ such that $\lVert f \rVert_{H_2^s(\sphere^d)} < \infty$.
One has
\begin{equation*}
	\lVert f \rVert_{H_2^s(\sphere^d)} \asymp \lVert f \rVert_{L^2(\sphere^d)} + \lVert (-\Delta)^{s/2} f \rVert_{L^2(\sphere^d)} \asymp \lVert (\Id_{L^2(\sphere^d)} - \Delta)^{s/2} f \rVert_{L^2(\sphere^d)}
\end{equation*}
and replacing the $L^2(\sphere^d)$-norm by another $L^p(\sphere^d)$-norm can be used to define fractional Sobolev spaces $H_p^s(\sphere^d)$ for general $p \in [1,\infty]$.
The Fourier multiplier approach presented here to define fractional Sobolev spaces is sufficient for the purposes of this paper.
This approach is equivalent to the one that defines Sobolev spaces in terms of charts \cite{Triebel1986}.
The first specific function classes for which asymptotic equivalence results will be derived in the sequel are the Sobolev balls
\begin{equation*}
	H^s_2(\sphere^d,R) = \{ f \in H^s_2(\sphere^d) : \lVert f \rVert_{H_2^s(\sphere^d)} \leq R \},
\end{equation*}
consisting of functions in $H^s_2(\sphere^d)$ with Sobolev norm bounded by $R\geq 0$.

In order to derive from Theorem~\ref{thm:fixed:general} asymptotic equivalence of $\Fexp_n^d$ and $\Gexp_n^d$ over Sobolev balls, we take $\{ \psi_1,\ldots,\psi_m \}$ as the set of all spherical harmonics up to a certain resolution level $L$, that is, we set
$\psi_j = Y_{\ell,m}$ if $j=\iota(\ell,m)$ for the enumeration function $\iota$ introduced in Section~\ref{subsec:spherical_harmonics}.
Consequently, $m = \dim \mathscr P^d_L \asymp L^d$.
In this specific setup, the assumption that the set $\XX$ of sampling points is a spherical $t$-design with $t \geq 2L$ ensures that $\ftilde_m$ defined in \eqref{eq:def:ftilde_m} is the least-squares approximation of $f$ in $\mathscr P_L^d$ based on the data $(x_1,f(x_1)),\ldots,(x_n,f(x_n))$, that is,
\begin{equation}\label{eq:def:ftilde_m:Sobolev}
	\ftilde_m = \arg \min_{p \in \mathscr P_L^d} ~ \sum_{i=1}^{n} (f(x_i) - p(x_i))^2.
\end{equation}
In fact, this follows from the inclusion $\mathscr P^d_L \cdot \mathscr P^d_L \subseteq \mathscr P^d_{2L}$, which implies that for the design matrix $X=(\psi_j(x_i)) \in \R^{n \times m}$ we have $X^\transpose X = n I_{m}$.
Hence, the least-squares approximation is given by $\sum_{j=1}^{m} \betatilde_{j} \psi_j$, where $\betatilde = (\betatilde_1,\ldots,\betatilde_m)^\transpose$ satisfies
\begin{align*}
	\betatilde &= (X^\transpose X)^{-1} X^\transpose (f(x_1),\ldots,f(x_n))^\transpose=n^{-1} X^\transpose (f(x_1),\ldots,f(x_n))^\transpose.
\end{align*}
From this, we obtain $\betatilde_j = n^{-1} \sum_{i=1}^{n} f(x_i) \psi_j(x_i) = \langle f,\psi_j\rangle_n$, showing that the least-squares approximation coincides with $\ftilde_m$ in this case.
Before we devote ourselves to the application of Theorem~\ref{thm:fixed:general} to Sobolev balls, we make two remarks.

\begin{remark}
	For the case of Sobolev balls, condition \eqref{eq:cond:themistoclakis} leads to a stronger assumption on the interplay of the sample size $n$, the resolution level $L$, and the smoothness parameter $s$ than necessary.
	More precisely, we will consider below spherical $t$-designs of cardinality $n \asymp L^d$, which is the minimal order achievable~\cite{Delsarte1977Spherical}.
	Then, global asymptotic equivalence can be achieved under the assumption $s > d/2$.
	From \eqref{eq:cond:themistoclakis}, only the more restrictive condition $s>d$ can be obtained.
	In fact, this is a direct consequence of Equation~(26) in \cite{Kushpel2000}, which implies that
	\begin{equation*}
		\sup_{f \in H^s_2(\sphere^d,R)} \inf_{p \in \mathscr P^d_L} \, \lVert f - p \rVert_{L^\infty(\sphere^d)} \gtrsim L^{-s+d/2}.
	\end{equation*}
	This lower bound implies that the condition $L^{-s+d} \to 0$ as $L = L(n) \to \infty$, which is equivalent to $s>d$, must be satisfied in order to guarantee that the term $\Delta_1$ converges to zero as desired.
	The less restrictive condition $s>d/2$, established below by finding an appropriate upper bound for the quantity $\lVert f-\Pi_{\Pscr^d_L} f \rVert_n$, is also the one expected from the corresponding results in the Euclidean case \cite{Reiss2008Asymptotic}.
\end{remark}

\begin{remark}\label{rem:tight_t-designs}
	Contrary to the approach in the Euclidean case considered in \cite{Reiss2008Asymptotic}, the least-squares approximation $\ftilde_m$ cannot in general be chosen as an interpolation through the points $(x_1,f(x_1)),\ldots,(x_n,f(x_n))$.
	Lemma~3 in \cite{Sloan1995Polynomial} states that the classical interpolation formula
	\begin{equation*}
		\ftilde_m(x_i) = f(x_i), \qquad i=1,\ldots,n,
	\end{equation*}
	holds for all continuous functions on the sphere if and only if the equal-weight cubature rule associated with the design $\XX$, which is {exact} for $f \in \mathscr P^d_{2L}$, is also \emph{minimal} in the sense that $n$ equals the dimension of $\mathscr P^d_L$ (see \cite{Sloan1995Polynomial}, p.~242, for this definition of minimality).
	Now, also from \cite{Sloan1995Polynomial}, Section~4.1, we report that spherical $t$-designs with this property (such spherical $t$-designs are also referred to as \emph{tight}) only exist in a few special cases.
	More precisely, in \cite{Bannai1979}, it is shown that tight spherical designs do not exist for $d \geq 2$ and $L \geq 3$.
	Therefore, the sample size $n$ is strictly larger than $\dim \mathscr P^d_L$ in general, and $\ftilde_m\lvert_{\XX} ~ \neq f\lvert_{\XX}$.
	Consequently, the term $\Delta_1$ in Theorem~\ref{thm:fixed:general} does not vanish.
	In the case $d=1$ considered in \cite{Reiss2008Asymptotic}, on the contrary, the term $\Delta_1$ is not present.
	In that case, a set of $2L+1$ equidistant design points (these points form a regular $2L+1$-gon on the unit circle $\sphere^1$; see also Example~2.6 in \cite{Bannai2009}) defines a spherical $t$-design for $t=2L$, and the least-squares approximation interpolates $f$ on this design.
	In the multivariate case with regression domain equal to $[0,1]^d$, which is also treated in \cite{Reiss2008Asymptotic}, a product design approach can be chosen, which inherits the interpolation property from the case $d=1$.
	In this sense, the term $\Delta_1$ is a new ingredient appearing in the bound on the Le Cam distance due to the considered spherical framework of dimension $d \geq 2$.
	The term $\Delta_2$, as mentioned above, is standard. In contrast to the work \cite{Reiss2008Asymptotic}, where fine properties of the Fourier basis are used in order to bound this term, we will exploit recent results from \cite{Lu2023} concerning the approximation by (weighted) least squares polynomials on the sphere to bound $\Delta_2$ (and also $\Delta_1$).
\end{remark}

Applying the bound on the Le Cam distance derived in Theorem~\ref{thm:fixed:general} establishes asymptotic equivalence of the experiments $\Fexp_n^d$ and $\Gexp_n^d$ over Sobolev balls.
\begin{theorem}\label{thm:fixed:Sobolev}
	Assume that $\XX = \{ x_1,\ldots,x_n \}$ is a spherical $t$-design with $t \geq 2L$.
	Then, for $\Theta=H_2^s(\sphere^d,R)$ with $s>d/2$,
	\begin{equation*}
		\Delta(\Fexp^d_{n},\Gexp^d_{n}) \lesssim \sigma^{-1} n^{1/2}L^{-s}R. 
	\end{equation*}
	In particular, choosing $\XX$ as a spherical $t$-design with $t=2L$ and cardinality $n$ of the minimal possible order $n \asymp L^d$, yields
	\begin{equation*}
		\Delta(\Fexp^d_{n},\Gexp^d_{n}) \lesssim \sigma^{-1}L^{-s+d/2}R \asymp \sigma^{-1}n^{-s/d+1/2}R,
	\end{equation*}
	and the experiments $\Fexp^d_{n}$ and $\Gexp^d_{n}$ are asymptotically equivalent as $n \to \infty$.
\end{theorem}

\subsubsection{Asymptotic equivalence over Besov balls}\label{sec:AE:Besov}

The representation of a regression function as a spherical harmonic expansion as in \eqref{eq:fourier:series:spherical:harmonics} suffers from the drawback that spherical harmonics are spread all over the sphere.
This leads to poor local performance of regression estimates relying on truncated spherical harmonic expansions.
In order to address this issue, the article \cite{Narcowich2006} introduced a class of localized tight frames on spheres of arbitrary dimension, which are referred to as \emph{needlets} due to their excellent localization properties.
In Figure~\ref{fig:sh_vs_needlet}, this essential difference between spherical harmonics and needlets is illustrated by opposing typical heatmaps of a spherical harmonic and a spherical needlet, respectively.
\begin{figure}[h]
	\centering
	\begin{subfigure}[b]{0.495\textwidth}
		\centering
		\includegraphics[width=\textwidth]{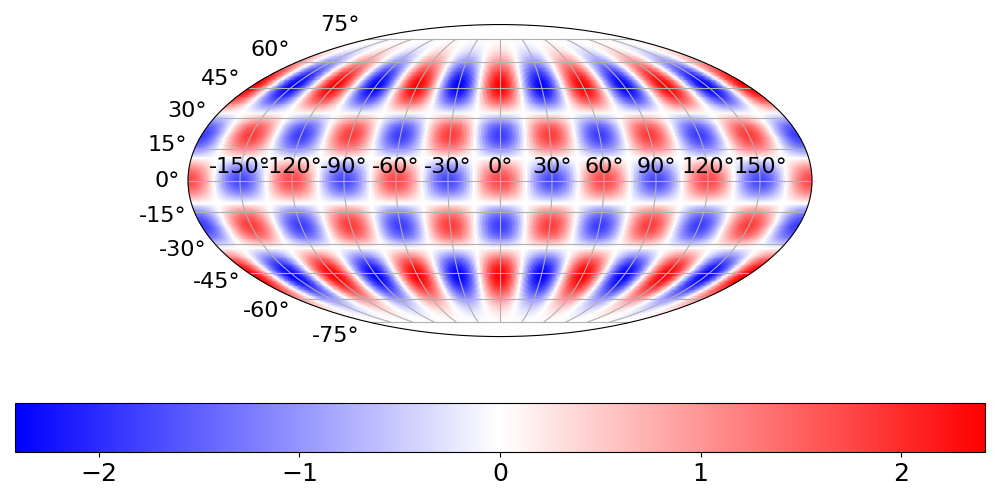}
	\end{subfigure}
	\hfill
	\begin{subfigure}[b]{0.495\textwidth}
		\centering
		\includegraphics[width=\textwidth]{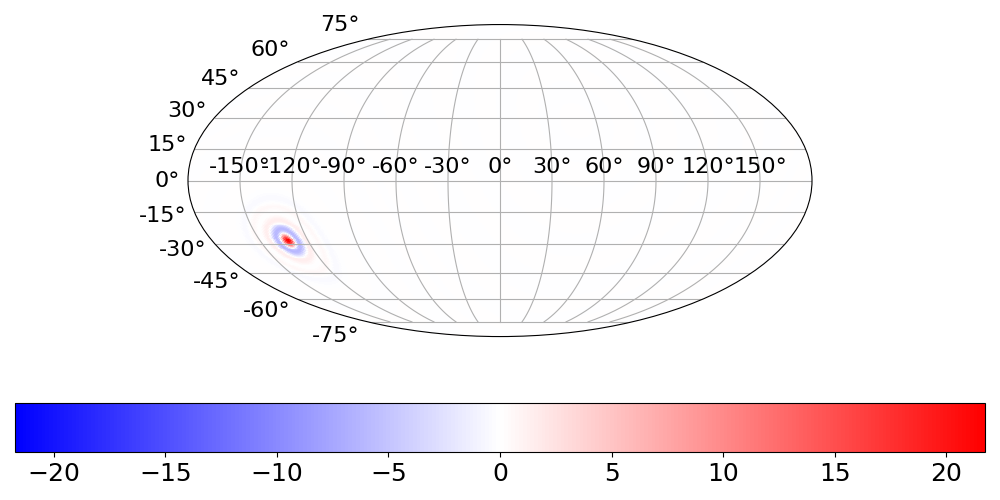}
	\end{subfigure}
	\medskip
	\caption{Heatmaps of a spherical harmonic from the eigenspace $\mathscr H^2_{10}$ (left plot) and a spherical needlet (right plot). The spherical needlet $\psi_{6,6141}$ is constructed using the filter function $h$ from Section~5 in \cite{Wang2017a} and is based on a (numerically determined) spherical $t$-design with $t=127$ as underlying cubature rule. This spherical design is available as file \texttt{ss127.08130} at \url{https://web.maths.unsw.edu.au/~rsw/Sphere/EffSphDes/ss.html}.}
	\label{fig:sh_vs_needlet}
\end{figure}

We give a brief sketch of the construction of the needlets since our arguments in the following rely on some fine properties of the needlet expansion of a function.
Our presentation here is mainly based on the papers \cite{Baldi2009a} and \cite{Wang2017a}.

The first ingredient in the definition of needlets is a continuous and compactly supported function $h \colon [0,\infty) \to [0,\infty)$, which is referred to as a \emph{filter}.
We assume that $h$ satisfies 
\begin{align}
	&h \in C^\infty, \quad \mathrm{supp}(h) = [1/2,2],\label{eq:h:prop:1}\\
	&h^2(t) + h^2(2t) = 1, \quad t \in [1/2,2].\label{eq:h:prop:2}
\end{align}
For some of the properties stated below, the assumption $h \in C^\infty$ is more restrictive than necessary.
Given \eqref{eq:h:prop:1}, \eqref{eq:h:prop:2} is equivalent to the \emph{partition of unity} property for $h^2$,
\begin{equation*}
	\sum_{j=0}^{\infty} h^2 \left( \frac{t}{2^j} \right) = 1, \quad t \geq 1,
\end{equation*}
see \cite{Wang2017a}, p.~294.
For a filter function $h$ and $\tau \geq 0$, we consider the filtered kernel $\kappa_{\tau,h}\colon \sphere^d \times \sphere^d \to \R$ defined by
\begin{equation}\label{eq:def:kappa}
	\kappa_{\tau,h}(x,y) = \begin{cases}
		1, & \text{if } 0 \leq \tau < 1,\\
		\sum_{\ell=0}^\infty h \left( \frac{\ell}{\tau} \right) N^d_\ell P_\ell^{(d+1)}(\langle x,y \rangle) , & \text{if } \tau \geq 1,
	\end{cases}
\end{equation}
where $\langle \boldsymbol{\cdot},\boldsymbol{\cdot} \rangle$ denotes the standard inner product on $\R^{d+1}$ and $P_\ell^{(d+1)}$ the normalized Gegenbauer polynomial defined by
\begin{equation*}
	P_\ell^{(d+1)}(t) = \frac{P_\ell^{(\frac{d-2}{2},\frac{d-2}{2})}(t)}{P_\ell^{(\frac{d-2}{2},\frac{d-2}{2})}(1)}
\end{equation*}
with $P_\ell^{(\alpha,\beta)}$ the Jacobi polynomial of degree $\ell$ for $\alpha,\beta > -1$.

The second ingredient in the construction of spherical needlets is given by cubature rules
\begin{equation*}
	\{ (\xi_{j,k}, w_{j,k}) \in \sphere^d \times(0,\infty) : j \in \N_0, k=1,\ldots,N_j \}
\end{equation*}
which are exact for functions $f \in \mathscr P^d_{2^{j+1}-1}$, that is,
\begin{equation*}
	\sum_{k=1}^{N_j} w_{j,k} f(\xi_{j,k}) = \int_{\sphere^d} f(x) \dd \mu(x), \quad f \in \mathscr P^d_{2^{j+1}-1}.
\end{equation*}
The weights of these cubature rules are assumed to be positive but not necessarily equal (but in principle, the underlying cubature rules can themselves be chosen as equal-weight spherical $t$-designs as in the right plot of Figure~\ref{fig:sh_vs_needlet}).
Then, for any $j \in \N_0$ and $k=1,\ldots,N_j$, spherical needlets $\psi_{j,k}\colon \sphere^d \to \R$ are, for $x \in \sphere^d$, defined by
\begin{equation*}
	\psi_{j,k}(x) = \sqrt{w_{j,k}} \, \kappa_{2^{j-1},h}(x,\xi_{j,k})
\end{equation*}
or, equivalently, by
\begin{align*}
	&\psi_{0,k}(x) = \sqrt{w_{0,k}},\\
	&\psi_{j,k}(x) = \sqrt{w_{j,k}} \sum_{\ell=0}^\infty h \left( \frac{\ell}{2^{j-1}} \right) N^d_\ell P_\ell^{(d+1)} (\langle x,\xi_{j,k} \rangle), \quad j = 1,2,\ldots
\end{align*}
Of course, this construction of needlets depends on both the cubature points and their corresponding weights.
However, it can be imposed that
\begin{equation*}
	N_j \asymp 2^{jd},
\end{equation*}
which will be assumed from now on.
By construction, the $\psi_{j,k}$ are band-limited but not orthogonal.
More precisely, $\psi_{j,k}$ is a polynomial of degree $2^j -1$ and it holds that
\begin{equation*}\label{eq:orthogonality:psi_j_eta}
	\lvert j-j' \rvert \geq 2 \qquad \Rightarrow \qquad \langle \psi_{j,k},\psi_{j',k'} \rangle_{L^2(\sphere^d)} = 0.
\end{equation*}
Moreover, for any $p \in [1,\infty]$,
\begin{equation*}
	\lVert \psi_{j,k} \rVert_{L^p(\sphere^d)} \asymp 2^{jd(1/2-1/p)}.
\end{equation*}
The needlet coefficients $\beta_{j,k}$ of a function $f\colon \sphere^d \to \R$ are defined as
\begin{equation*}
	\beta_{j,k} = \langle f, \psi_{j,k} \rangle_{L^2(\sphere^d)} = \int_{\sphere^d} f(x) \psi_{j,k}(x) \dd \mu(x),
\end{equation*}
and a function $f \in L^2(\sphere^d)$ can be represented as
\begin{equation*}
	f = \sum_{j = 0}^\infty \sum_{k=1}^{N_j} \beta_{j,k} \psi_{j,k}.
\end{equation*}
By a slight abuse of notation, we denote the truncated needlet expansion including only the $\psi_{j,k}$ with $j \leq J$ by $f_J$,
\begin{equation*}
	f_J = \sum_{j=0}^J \sum_{k=1}^{N_j} \beta_{j,k} \psi_{j,k},
\end{equation*}
which can equivalently (see \cite{Wang2017a} for details) be written as
\begin{equation}\label{eq:f_J:int:kernel}
	f_J(x) =  \int_{\sphere^d} f(y) \kappa_{2^{J-1},H}(x,y) \dd \mu (y),
\end{equation}
where the new filter $H\colon [0,\infty) \to [0,\infty)$ is defined in terms of the needlet filter $h$ as follows:
\begin{equation}\label{eq:def:H}
	H(t) = \begin{cases}
		1, & 0 \leq t < 1,\\
		h^2(t), & t \geq 1,
	\end{cases}
\end{equation}
see \cite{Wang2017a}, Equation~(23).
Membership in spherical Besov spaces can be characterized by means of the needlet coefficients of a function $f$.
More precisely, $f$ belongs to the Besov space $B^s_{r,q}(\sphere^d)$ if and only if
\begin{equation*}
	\lVert f \rVert_{B_{r,q}^s(\sphere^d)} \defeq \lVert (2^{j(s-d/r+d/2)} \lVert (\beta_{j,k})_{k=1}^{N_j} \rVert_{\ell^r} )_{j = 0}^\infty \rVert_{\ell^q} < \infty.
\end{equation*}
$\lVert \,\boldsymbol\cdot\, \rVert_{B_{r,q}^s(\sphere^d)}$ is referred to as the Besov norm.
Putting
\begin{equation}\label{def:E_L}
	E_L(f,r) = \inf_{p \in \Pscr_L^d} \lVert f - p \rVert_{L^r(\sphere^d)},
\end{equation}
we have the equivalence
\begin{equation}\label{eq:equivalent:Besov:norm}
 	\lVert f \rVert_{B_{r,q}^s(\sphere^d)} \asymp \lVert f \rVert_{L^r(\sphere^d)} + \left( \sum_{j=0}^{\infty} (2^{js} E_{2^j}(f,r) )^q \right)^{1/q}
\end{equation}
of norms (see \cite{Narcowich2006a}, Chapter~5).
The Sobolev spaces $H^s_2(\sphere^d)$ introduced above coincide with the Besov spaces $B_{2,2}^s(\sphere^d)$ \cite{Baldi2009a}.
We denote by $B_{r,q}^s(\sphere^d,R)$ the ball of radius $R$ in the Besov space $B_{r,q}^s(\sphere^d)$.
We now apply the bound on the Le Cam distance derived in Theorem~\ref{thm:fixed:general} to establish asymptotic equivalence of the experiments $\Fexp^d_n$ and $\Gexp_n^d$ over such Besov balls.
Denote with $\ftilde_J$ the empirical needlet approximation defined by
\begin{align*}
	\ftilde_J = \sum_{j=0}^{J} \sum_{k=1}^{N_j} \betatilde_{j,k} \psi_{j,k},
\end{align*}
where
\begin{align*}
	\betatilde_{j,k} = \langle f,\psi_{j,k} \rangle_n = \frac{1}{n} \sum_{i=1}^{n} f(x_i)\psi_{j,k}(x_i)
\end{align*}
for the set $\XX = \{ x_1,\ldots,x_n \}$ of sampling points.
In the following, we assume that $\XX$ is a spherical $t$-design for $t \geq 2^{J+1}$ (this assumption is analogous to the one in Theorem~\ref{thm:fixed:general}).
Of course, $\ftilde_J$ is the empirical counterpart of $f_J$ obtained by replacing the $L^2(\sphere^d)$-inner product with its empirical analogue relying on the equal-weight cubature rule associated with the spherical design $\XX$.
Equivalently, we can write
\begin{align}
	\ftilde_J(x) &= \frac{1}{n} \sum_{j = 0}^J \sum_{i=1}^n f(x_i) \kappa_{2^{j-1},h^2}(x_i,x)\notag \\
	&= \frac 1 n \sum_{i=1}^n f(x_i) \kappa_{2^{J-1},H}(x_i,x),\label{eq:ftilde:kernel:H}
\end{align}
see \cite{Wang2017a}, Equations (36), (37), and (43).
The maximal resolution level $J=J(n)$ will be chosen such that
\begin{equation*}\label{eq:choice:J}
	2^J \asymp n^{1/d},
\end{equation*}
which coincides, except for a missing logarithmic term, with the choice of this truncation parameter in adaptive non-parametric estimation using needlets \cite{Baldi2009a,Monnier2011}.

\begin{theorem}\label{thm:fixed:Besov}
Assume that $\XX = \{ x_1,\ldots,x_n \}$ is a spherical $t$-design with $t \geq 2^{J+1}$, $J \in \N_0$.
Then, for $\Theta=B^s_{r,q}(\sphere^d,R)$ with $s>d/r$,
\begin{equation*}
	\Delta(\Fexp^d_{n},\Gexp^d_{n}) \lesssim \sigma^{-1} n^{1/2}2^{-J(s-d/r+d/2)}. 
\end{equation*}
In particular, choosing $\XX$ as a spherical $t$-design with $t= 2^{J+1}$ and cardinality $n$ of the minimal possible order $n \asymp 2^{Jd}$, yields
\begin{equation*}
	\Delta(\Fexp^d_{n},\Gexp^d_{n}) \lesssim \sigma^{-1} 2^{-J(s-d/r)} \asymp \sigma^{-1} n^{-s/d+1/r},
\end{equation*}
and the experiments $\Fexp^d_{n}$ and $\Gexp^d_{n}$ are asymptotically equivalent as $n \to \infty$.
\end{theorem}

\section{Regression on random uniform designs}\label{s:random:design}

The aim of this section is to establish asymptotic equivalence between the Gaussian white noise model $\Gexp_n^d$ given by \eqref{eq:general:wn} and random uniform design regression with model equations
\begin{equation}\label{eq:regression:random:design}
	Z_i = f(X_i) + \sigma \epsilon_i, \quad i=1,\ldots,n,
\end{equation}
where $X_1,\ldots,X_n$ are i.i.d.\ $\sim \Uc(\sphere^d)$, but all other quantities are defined as in the fixed design regression model $\Fexp_n^d$ defined by \eqref{eq:general:regression}.
In this section, we restrict ourselves to the smoothness class $\Theta = H^s_2(\sphere^d,R)$.
The corresponding statistical experiment is denoted by $\Rexp_n^d=\Rexp_n^d(\Theta)$, and the following result states the asymptotic equivalence of $\Rexp_n^d$ and $\Gexp_n^d$ under the same smoothness assumption $s>d/2$ as in the fixed design case.

\begin{theorem}\label{thm:random:Sobolev}
	Let $\Theta=H_2^s(\sphere^d,R)$ with $s>d/2$.
	Consider the random design regression experiment $\Rexp_n^d$ defined by \eqref{eq:regression:random:design}.
	For all sufficiently large $n \in \N$, let $L \in \N$ be maximal such that
	\begin{equation*}
		\dim \Pscr^d_L \leq  \frac{\kappa_2n}{\ln(n)},
	\end{equation*}
	where $\kappa_2 = (3\ln(1.5)-1)/6 \approx 0.036$.
	Then, for any $L_0 \in \N$ with $L_0 \leq L$, we have 
	\begin{equation*}
		\Delta(\Rexp^d_{n},\Gexp^d_{n}) \lesssim n^{-1/2} L_0^{d} + \sigma^{-1} L_0^{d/2-s}R + \sigma^{-1}n^{1/2}L^{-s}R + n^{-2}.
	\end{equation*}
	In particular, choosing $L_0=L_0(n) \leq L$ such that
	\begin{equation*}
		L_0 \to \infty \quad \text{and} \quad n^{-1/2}L_0^{d} \to 0
	\end{equation*}
	implies that
	\begin{equation*}
		\Delta(\Rexp^d_{n},\Gexp^d_{n}) \to 0,
	\end{equation*}
	and the experiments $\Rexp^d_{n}$ and $\Gexp^d_{n}$ are asymptotically equivalent as $n \to \infty$.
\end{theorem}
The proof of Theorem~\ref{thm:random:Sobolev}, which is given in Section~\ref{s:proof:thm:random:Sobolev}, is (similar to the one of Theorem~\ref{thm:fixed:general}) constructive in the sense that it provides explicit Markov kernels that allow one to transform observations between the considered experiments.
The proof relies on separate calculations for low- and high-frequency coefficients and is similar to the proof of the analogous result in the multivariate Euclidean case given in \cite{Reiss2008Asymptotic}. However, several additional technical tools (a generalization of the classical $QR$ decomposition, results from representation theory like Schur's lemma, reverse Hölder inequalities for spherical harmonics, a Taylor series expansion of the matrix-valued function $A \mapsto (I+A)^{1/2}$) are necessary. The interaction of these tools is of independent interest and might be useful to establish further extensions of Theorem~\ref{thm:random:Sobolev} in future work.

\section{Asymptotic non-equivalence}\label{sec:non-equivalence}

We now address the sharpness of the asymptotic equivalence results from Section~\ref{s:fixed:design} by showing that the conditions $s > d/2$ (for Sobolev balls) and $s>d/r$ (for Besov balls) cannot be further weakened.
This result is not restricted to the case where the design points form a spherical $t$-design but holds for any choice of design points fixed in advance.
It will turn out sufficient to reduce to the case where the common parameter space $\Theta$ for both the regression and the Gaussian white noise experiment consists of two elements only.
More precisely, we construct subsets $\Theta_n' \subset \Theta$ with $\lvert \Theta_n' \rvert = 2$ such that the observations in the regression model become indistinguishable on $\Theta_n'$, whereas in the Gaussian white noise model the total variation distance between the two potential distributions is uniformly bounded from below by a positive constant independent of the sample size $n$.
From this, asymptotic non-equivalence of the two experiments follows by a standard argument.
The construction of the hypotheses in the two experiments is inspired by the approach in \cite{Wang2016}, where optimal recovery of smooth functions on spheres from function values was studied (especially, the ideas used in the proof of Lemma~3.5 of that reference turn out to be useful for the proof of the following result).

\begin{theorem}\label{thm:non-equivalence}
	Consider $\Theta = H_2^s(\sphere^d,R)$ with $s \leq d/2$ or $\Theta = B^s_{r,q}(\sphere^d, R)$ with $s \leq d/r$.
	Let $\XX$ be any point set with $\lvert \XX \rvert \asymp n$ (not necessarily a spherical $t$-design).
	Denote with $\Gexp_n^d$ the Gaussian white noise experiment given by \eqref{eq:general:wn} and with $\Fexp_n^d=\Fexp_n^d(\Theta,\XX)$ the regression experiment defined by \eqref{eq:general:regression} on the design $\XX$ with regression function $f\in \Theta$.
	Then,
	\begin{equation}\label{eq:bound:non-equivalence}
		\Delta(\Fexp_n^d,\Gexp_n^d) \geq \delta(\Fexp_n^d,\Gexp_n^d) \gtrsim 1.
	\end{equation}
	In particular, the experiments $\Fexp_n^d$ and $\Gexp_n^d$ are \emph{not} asymptotically equivalent.
\end{theorem}

\begin{remark}
The statement of Theorem~\ref{thm:non-equivalence} is in agreement with usual embedding theorems which state that $H_2^s(\sphere^d)$ and $B^s_{r,q}(\sphere^d)$ can be continuously embedded into the space $C(\sphere^d)$ of continuous functions on the sphere provided that $s>d/2$ and $s>d/r$, respectively (see \cite{Narcowich2006a} and \cite{Baldi2009a} for details).
In this light, the requirements $s>d/2$ and $s>d/r$ are natural for any method based on function values. The proof of Theorem~\ref{thm:non-equivalence} even shows that asymptotic equivalence does not hold if one restricts the function spaces to functions having a finite series expansion (in the proof, we consider $\Theta_n' \subset \Pscr^d_L$ for some $L=L(n) \to \infty$).
\end{remark}

\section{Discussion}\label{sec:discussion}

We have proven that non-parametric regression with spherical regressors is asymptotically equivalent, in the sense of Le Cam, to a corresponding Gaussian white noise experiment.
The results hold for both the fixed design case, where the sampling points form a spherical $t$-design, and the random uniform design case.
As special cases of function classes for which this asymptotic equivalence holds Sobolev and Besov balls were considered, and the smoothness assumptions imposed to establish asymptotic equivalence over these function spaces were shown to be sharp in the fixed design case.

The derived results suggest that both spherical $t$-designs and random uniform designs are a good choice of sampling points in non-parametric regression as the resulting statistical experiment is asymptotically equivalent to the Gaussian white noise model, which is regarded as the simplest model of the form $\texttt{data} = \texttt{signal} + \texttt{noise}$.
The established asymptotic equivalences suggest that known results for the Gaussian white noise model on the sphere (for instance, the sharp minimax results from \cite{Klemelae1999Asymptotic} already cited in the introduction) are also valid for both the fixed and the random uniform design regression framework with spherical regressors.
Besides estimation procedures as considered in \cite{Klemelae1999Asymptotic}, (sharp) non-parametric testing rates and confidence bands might now also be transferable from the idealized Gaussian white noise model to the more realistic regression models.

The interpretation of the considered designs as a good choice of sampling points is in line with existing results on optimal designs in finite-dimensional linear regression models.
For truncated spherical harmonic expansions, the papers \cite{Dette2005} (for the two-dimensional sphere) and \cite{Dette2019} (for spheres of arbitrary dimension) identify the uniform distribution as an optimal design distribution with respect to the $\Phi_p$-criteria introduced by Kiefer \cite{Kiefer1974}.
Such a design given by an absolutely continuous distribution cannot be implemented directly in practice.
In \cite{Dette2019}, Remark~3.1, the authors mention spherical $t$-designs as a potential remedy to address this issue.
Even more recently, \cite{Haines2024} discusses the use of spherical $t$-designs as optimal designs for the special case of the two-dimensional sphere.
So-called $\lambda$-designs, which extend the notion of spherical $t$-designs to Riemannian manifolds, are identified as optimal designs for regression on Lie groups in \cite{Chakraborty2026}.

In this work, we have restricted ourselves to the case of spheres of arbitrary dimension for two reasons: (i) the restriction to this special case keeps the technical jargon to a minimum but already incorporates all the essential ideas, (ii) the case of regression with spherical regressors, especially for the two-dimensional sphere $\sphere^2$, is certainly the most relevant one in applications.
We conjecture that the main parts of our analysis can be carried over to non-parametric regression on more general compact Riemannian manifolds.
Regression models on manifolds and the presumably equivalent Gaussian white noise model are, for instance, considered in \cite{Castillo2014} in the context of non-parametric Bayesian estimation.
Another work in this direction is \cite{Kerkyacharian2011}, where confidence bands for needlet density estimators on compact homogeneous manifolds have been derived.



\section{Proofs}\label{sec:proofs}

\subsection{Proof of Theorem~\ref{thm:fixed:general}}

We first derive a uniform upper bound on $\Delta(\Fexp_{n}^d,\widetilde\Gexp^d_{n})$, that is, the Le Cam distance between the regression experiment $\Fexp_{n}^d$ and the intermediate experiment $\widetilde\Gexp^d_{n}$.
For this, we first consider the deficiency $\delta(\Fexp_{n}^d,\widetilde\Gexp^d_{n})$.
Assume that an observation $Z = (Z_1,\ldots,Z_n)^\transpose$ from the regression experiment is given.
Setting $\langle v,\psi_j \rangle_n=n^{-1} \sum_{i=1}^n v_i \psi_j(x_i)$ for $v=(v_1,\ldots,v_n)^\top \in \R^n$, we define the random process
\begin{align*}
	\widetilde Z_m = \sum_{j=1}^{m} \langle Z,\psi_j \rangle_n \psi_j &= \sum_{j=1}^{m} \langle f,\psi_j \rangle_n \psi_j + \sigma \sum_{j=1}^{m} \langle \epsilon,\psi_j \rangle_n \psi_j= \ftilde_m + \frac{\sigma}{\sqrt n}\,\zeta_m.
\end{align*}
The random variables $\langle \epsilon,\psi_j \rangle_n = n^{-1}\sum_{i=1}^{n} \epsilon_i \psi_j(x_i)$ are Gaussian with mean zero.
Combining the inclusion $S \cdot S \subseteq \mathscr P^d_{L} \cdot \mathscr P^d_{L} \subseteq \mathscr P^d_{2L}$ with the assumption that $\XX$ is a spherical $t$-design for $t \geq 2L$ implies that
\begin{equation*}
	\Eb [\langle \epsilon,\psi_j \rangle_n \langle \epsilon,\psi_k \rangle_n] = n^{-1} \langle \psi_j,\psi_k \rangle_{L^2(\sphere^d)}.
\end{equation*}
By linearity, this shows that the process $\zeta_m$ is standard Gaussian white noise on $S$.
By adding Gaussian white noise, scaled by the same factor $\sigma/\sqrt n$, on the $L^2(\sphere^d)$-orthogonal complement of $S$, we obtain the process
\begin{equation*}
	\widetilde Z = \ftilde_m + \frac{\sigma}{\sqrt n} \zeta,
\end{equation*}
where $\zeta$ denotes a standard Gaussian white noise process.
In differential notation, the process $\widetilde Z$ can equivalently be written as
\begin{equation*}
	\dd \widetilde Z(x) = \ftilde_m(x) \dd \mu(x) + \frac{\sigma}{\sqrt n} \dd W(x),
\end{equation*}
showing that
\begin{equation}\label{eq:bound:deficiency:regression:intermediate}
	\delta(\Fexp_{n}^d,\Gtildeexp^d_{n}) = 0.
\end{equation}

Conversely, given the continuous observation \eqref{eq:intermediate} from the intermediate experiment and choosing an $L^2(\sphere^d)$-orthonormal basis $\{ \varphi_1,\ldots,\allowbreak\varphi_{\dim S}\}$ of $S$,
the random vector $\thetahat = (\thetahat_{1},\ldots,\thetahat_{\dim S})^\top$ with
\begin{equation*}
	\thetahat_{j} = \int_{\sphere^d} \varphi_j(x)\dd \widetilde Z(x), \quad j=1,\ldots,\dim S,
\end{equation*}
follows a multivariate Gaussian distribution with mean
\begin{equation*}
	\left( \int \ftilde_m(x)\varphi_1(x)\dd\mu(x),\ldots,\int \ftilde_m(x)\varphi_{\dim S}(x)\dd \mu(x) \right)^\transpose
\end{equation*}
and covariance matrix $\sigma^2 n^{-1} I_{\dim S}$.
Consider the $n \times \dim S$ design matrix $X=(\varphi_j(x_i))$.
The vector $\widehat Z = (\widehat Z_1,\ldots,\widehat Z_n)^\top \defeq X \thetahat$ of fitted values at the sampling points $x_1,\ldots,x_n$ follows a multivariate Gaussian with mean vector
\begin{equation}\label{eq:evalations:hyperinterpolant}
	(\ftilde_m(x_1),\ldots,\ftilde_m(x_n))^\transpose
\end{equation}
and covariance matrix $\sigma^2n^{-1}XX^\top$.
Consider a mean zero multivariate Gaussian random vector $\xi$ with covariance matrix
\begin{equation*}
	\Cov(\xi)= \sigma^2 (I_n - X(X^\top X)^{-1}X^\top) = \sigma^2 (I_n - n^{-1} XX^\top),
\end{equation*}
independent of $\widehat Z$ (the matrix $\sigma^2 (I_n - X(X^\top X)^{-1}X^\top)$ corresponds to the covariance matrix of the residuals in the linear model given by $Y=X\theta+\epsilon$ and is therefore positive semi-definite; the fact that the design points form a spherical $t$-design implies that $(X^\top X)^{-1}=n^{-1}I_{\dim S}$).
It follows that the vector $Z' = \widehat Z + \xi$ follows a multivariate Gaussian with mean \eqref{eq:evalations:hyperinterpolant}
and covariance matrix $\sigma^2 I_n$.
Hence, the deficiency $\delta(\Gtildeexp^d_{n},\Fexp_{n}^d)$ can be bounded by taking the supremum over all $f$ of the total variation distance between the multivariate Gaussian vectors $Z$ and $Z'$,
\begin{align}
	\delta(\Gtildeexp^d_{n},\Fexp_{n}^d) &\leq \sup_{f \in \Theta} \left[ 1 - 2\Phi \left( -\frac{\sqrt n}{2\sigma}\, \lVert f- \ftilde_m \rVert_n \right)\right],\label{eq:bound:deficiency:intermediate:regression}
\end{align}
see, for instance, \cite{Nielsen2024}.
Combining \eqref{eq:bound:deficiency:regression:intermediate} and \eqref{eq:bound:deficiency:intermediate:regression} yields
\begin{equation}\label{eq:Le:Cam:regression:intermediate}
	\Delta(\Fexp_{n}^d,\Gtildeexp^d_{n}) \leq 1 - 2\Phi \left( -\frac{\sqrt n}{2\sigma} \sup_{f \in \Theta} \, \lVert f- \ftilde_m \rVert_n \right).
\end{equation}
Next, we derive a uniform upper bound on the Le Cam distance $\Delta(\Gtildeexp^d_{n},\Gexp^d_{n})$ between the intermediate experiment $\Gtildeexp^d_{n}$ and the Gaussian white noise experiment $\Gexp^d_{n}$.
Denote by $\Pb_f$ the distribution of the process $Z$ in the experiment $\Gexp^d_{n}$, and by $\Pb_{\ftilde_m}$ the distribution of the process $\widetilde Z$ in the experiment $\Gtildeexp^d_{n}$, respectively.
Then, \eqref{eq:app:bound:tv:gwnm} yields
\begin{equation*}
	V(\Pb_f,\Pb_{\ftilde_m}) = 1 - 2 \Phi \left( - \frac{\sqrt n}{2\sigma}\,\lVert f - \ftilde_m \rVert_{L^2(\sphere^d)} \right),
\end{equation*}
from which it follows that
\begin{equation}\label{eq:Le:Cam:intermediate:WN}
	\Delta(\Gtildeexp^d_{n},\Gexp^d_{n}) \leq 1 - 2 \Phi  \left( - \frac{\sqrt n}{2\sigma} \sup_{f \in \Theta}\, \lVert f - \ftilde_m \rVert_{L^2(\sphere^d)} \right).
\end{equation}
By combining \eqref{eq:Le:Cam:regression:intermediate} and \eqref{eq:Le:Cam:intermediate:WN}, the bound for $\Delta(\Fexp^d_{n},\Gexp^d_{n})$ announced in the theorem follows from the triangle inequality for the Le Cam distance.

\subsection{Proof of Theorem~\ref{thm:fixed:Sobolev}}

In the following, we separately treat the terms $\Delta_1$ and $\Delta_2$ from Theorem~\ref{thm:fixed:general}.
Recall that we assume that $\XX = \{ x_1,\ldots,x_n \}$ is a spherical $t$-design with $t \geq 2L$.

\subsubsection*{Bound for $\Delta_1$ when $\Theta=H_2^s(\sphere^d,R)$}

With $\ftilde_m$ defined by \eqref{eq:def:ftilde_m:Sobolev} we have
\begin{align*}
	\lVert f- \ftilde_m \rVert_n \leq \lVert f-\Pi_{\Pscr^d_L} f\rVert_n.
\end{align*}
Putting
\begin{equation*}
	A_j f = \Pi_{\Pscr^d_{2^j L}}f - \Pi_{\Pscr^d_{2^{j-1} L}}f
\end{equation*}
for $j \in \N$, we can write
\begin{equation*}
	f(x) - \Pi_{\Pscr^d_L} f(x) = \sum_{j=1}^{\infty} A_j f(x),
\end{equation*}
and the series on the right-hand side converges uniformly for $f \in H_2^s(\sphere^d,R)$ when $s > d/2$.
Using Lemma~\ref{lemma:Dai06} with $p_0=p=2$ (note that $A_jf \in \Pscr^d_{2^j L}$ by definition), we obtain for $f \in H_2^s(\sphere^d,R)$ that
\begin{align*}
	\lVert A_j f \rVert_n^2 &= \frac{1}{n}\sum_{i=1}^n \lvert A_j f(x_i) \rvert^2\lesssim 2^{jd} \int_{\sphere^d} \lvert A_j f(x)\rvert^2 \dd \mu(x)\lesssim 2^{jd}2^{-2js}L^{-2s}R^2.
\end{align*}
As a consequence, we obtain by means of the triangle inequality that, for $s>d/2$,
\begin{align*}
	\lVert f - \Pi_{\Pscr^d_L}f \rVert_n \leq \sum_{j=1}^{\infty} \lVert A_j f \rVert_n\lesssim L^{-s}R \sum_{j=1}^{\infty} 2^{-j(s-d/2)} \lesssim L^{-s}R.
\end{align*}
It follows that for $\Theta=H_2^s(\sphere^d,R)$, the term $\Delta_1$ in Theorem~\ref{thm:fixed:general} can be bounded by
\begin{equation}\label{eq:bound:Delta_1:Sobolev}
	\Delta_1 \lesssim \sigma^{-1} n^{1/2} L^{-s}R.
\end{equation}

\subsubsection*{Bound for $\Delta_2$ when $\Theta=H_2^s(\sphere^d,R)$}

Under the assumptions of Theorem~\ref{thm:fixed:general}, applying \cite{Lu2023}, Theorem~1.2, Equation~(1.6), yields that, for $\Theta=H_2^s(\sphere^d,R)$ with $s>d/2$,
\begin{align}
	\Delta_2 &\leq 1 - 2\Phi  \left( - \frac{\sqrt n}{2\sigma} \sup_{f \in \Theta} \, \lVert f - \ftilde_m \rVert_{L^2(\sphere^d)} \right)\notag\\[2mm]
	&\lesssim \sigma^{-1} n^{1/2} \sup_{f \in \Theta} \, \lVert f - \ftilde_m \rVert_{L^2(\sphere^d)}\notag\\
	&\lesssim \sigma^{-1} n^{1/2}L^{-s}R.\label{eq:bound:Delta_2:Sobolev}
\end{align}
Combining the bounds \eqref{eq:bound:Delta_1:Sobolev} and \eqref{eq:bound:Delta_2:Sobolev} implies the statement of Theorem~\ref{thm:fixed:Sobolev}.

\subsection{Proof of Theorem~\ref{thm:fixed:Besov}}

As for the Sobolev case, the proof of Theorem~\ref{thm:fixed:Besov} for the Besov case relies on finding suitable upper bounds for the quantities $\Delta_1$ and $\Delta_2$ in Theorem~\ref{thm:fixed:general}.
In the following analysis, we assume that $r \leq 2$. The result of the theorem equally holds true for the case $r \geq 2$ and directly follows from the case $r=2$ via the Besov embedding $B^s_{r,q}(\sphere^d) \subseteq B^s_{2,q}(\sphere^d)$ for $r \geq 2$ (see \cite{Baldi2009a}, Theorem~5).

\subsubsection*{Bound for $\Delta_1$ when $\Theta=B_{r,q}^s(\sphere^d,R)$}

With $\ftilde_J$ taking the role of $\ftilde_m$ in Theorem~\ref{thm:fixed:general}, that is, setting $m=\sum_{j=0}^J N_j$, we obtain
\begin{align*}
	\Delta_1 \leq \sigma^{-1} n^{1/2} \sup_{f \in \Theta} \, \lVert f-\ftilde_J \rVert_n.
\end{align*}
To bound $\lVert f-\ftilde_J \rVert_n$ uniformly for $f \in \Theta= B_{r,q}^s(\sphere^d,R)$, we use the decomposition
\begin{align*}
	\lVert f-\ftilde_J \rVert_n &= \lVert f_J -\ftilde_J + f - f_J \rVert_n\\
	&\leq \lVert f_J -\ftilde_J \rVert_n + \lVert f - f_J \rVert_n\\
	&= \lVert f_J -\ftilde_J \rVert_{L^2(\sphere^d)} + \lVert f - f_J \rVert_n\\
	&\leq \lVert f -\ftilde_J \rVert_{L^2(\sphere^d)} + \lVert f-f_J \rVert_{L^2(\sphere^d)} + \lVert f - f_J \rVert_n,
\end{align*}
where we have used that $\XX$ is a spherical $t$-design with $t \geq 2 \cdot (2^J -1)$.
Below, in the analysis of $\Delta_2$, we will show that
\begin{equation}\label{eq:Besov:Delta1:1}
	\lVert f - \ftilde_J \rVert_{L^2(\sphere^d)} \lesssim 2^{-J(s-d/r+d/2)}.
\end{equation}
From \cite{Baldi2009a}, p.~3383, we obtain
\begin{equation}\label{eq:Besov:Delta1:2}
	\lVert f-f_J \rVert_{L^2(\sphere^d)} \lesssim 2^{-J(s-d/r+d/2)}.
\end{equation}
Finally, by Lemma~\ref{lemma:Dai06} and again the estimate from \cite{Baldi2009a}, p.~3383, we obtain that
\begin{align}
	\lVert f-f_J \rVert_n &\lesssim \sum_{j > J} \left( \left( \frac{2^j}{2^J} \right)^d \int_{\sphere^d} \lvert \sum_{k=1}^{N_j} \beta_{j,k} \psi_{j,k}(x) \rvert^2 \dd \mu(x) \right)^{1/2}\notag\\
	&= 2^{-Jd/2} \sum_{j > J}  2^{jd/2} \lVert \sum_{k=1}^{N_j} \beta_{j,k} \psi_{j,k} \rVert_{L^2(\sphere^d)}\notag\\
	&\lesssim 2^{-Jd/2} \sum_{j > J}  2^{jd/2} \cdot 2^{-j(s-d/r+d/2)}\notag\\
	&= 2^{-Jd/2} \sum_{j > J}  2^{-j(s-d/r)}\notag\\
	&\asymp 2^{-J(s-d/r+d/2)}.\label{eq:Besov:Delta1:3}
\end{align}
Combining \eqref{eq:Besov:Delta1:1}--\eqref{eq:Besov:Delta1:3}, we obtain
\begin{equation}\label{eq:Besov:Delta1}
	\Delta_1 \lesssim \sigma^{-1} n^{1/2} 2^{-J(s-d/r+d/2)}.
\end{equation}

\subsubsection*{Bound for $\Delta_2$ when $\Theta=B_{r,q}^s(\sphere^d,R)$}

We now derive a bound for $\lVert f - \ftilde_J \rVert_{L^2(\sphere^d)}$.
Using \eqref{eq:ftilde:kernel:H}, we obtain
\begin{align*}
	f(x) - \ftilde_J(x) &= \sum_{j = 0}^\infty \sum_{k=1}^{N_j} \beta_{j,k}\psi_{j,k}(x) - \frac{1}{n} \sum_{i=1}^n f(x_i) \kappa_{2^{J-1},H}(x_i,x)\\
	&=\sum_{j>J}^\infty \sum_{k=1}^{N_j} \beta_{j,k} \psi_{j,k}(x)\\
	&\phantom{=}- \sum_{j>J}^\infty \frac{1}{n} \sum_{i=1}^n \left( \sum_{k=1}^{N_j} \beta_{j,k} \psi_{j,k}(x_i) \right) \kappa_{2^{J-1},H}(x_i,x),   
\end{align*}
where we use both the fact that the equal-weight cubature rule associated with the spherical $t$-design $\XX$ is exact for polynomials of degree $\leq 2^{J+1}$ and identity \eqref{eq:f_J:int:kernel}.
It follows that
\begin{align*}
	\lVert f - \ftilde_J \rVert_{L^2(\sphere^d)} &\leq \sum_{j > J} \left\| \sum_{k=1}^{N_j} \beta_{j,k} \psi_{j,k} \right\|_{L^2(\sphere^d)}\notag\\
	&\phantom{=}+ \sum_{j > J} \left\|  \frac{1}{n} \sum_{i=1}^n \left( \sum_{k=1}^{N_j} \beta_{j,k} \psi_{j,k}(x_i) \right) \kappa_{2^{J-1},H}(x_i,\boldsymbol{\cdot}) \right\| _{L^2(\sphere^d)}.\label{eq:Besov:Delta2:dec}
\end{align*}
From \cite{Baldi2009a}, p.~3383, we obtain that the first term on the right-hand side can be bounded by
\begin{equation}\label{eq:Besov:Delta2:1}
	\sum_{j > J} \left\|  \sum_{k=1}^{N_j} \beta_{j,k} \psi_{j,k} \right\|_{L^2(\sphere^d)} \lesssim 2^{-J(s-d/r+d/2)}.
\end{equation}
Let us now consider the second term.
\cite{Wang2017a}, Theorem~3.3, yields that for $H \in C^\infty$ the inequality
\begin{equation*}
	\lVert \kappa_{\tau,H}(x,\boldsymbol{\cdot}) \rVert_{L^1(\sphere^d)} \leq C
\end{equation*}
holds with a constant $C = C(d,H)$ that depends on neither $\tau$ nor $x$.
We obtain
\begin{align*}
	&\left\|  \frac{1}{n} \sum_{i=1}^n \left( \sum_{k=1}^{N_j} \beta_{j,k} \psi_{j,k}(x_i) \right) \kappa_{2^{J-1},H}(x_i,\boldsymbol{\cdot}) \right\|_{L^2(\sphere^d)}^2\\
	&= \int_{\sphere^d} \bigg\lvert \frac{1}{n} \sum_{i=1}^n \left( \sum_{k=1}^{N_j} \beta_{j,k} \psi_{j,k}(x_i) \right) \kappa_{2^{J-1},H}(x_i,x) \bigg\rvert^2 \dd \mu(x)\\[2mm]
	&\leq \int_{\sphere^d} \left( \frac{1}{n} \sum_{i=1}^{n} \lvert \sum_{k=1}^{N_j} \beta_{j,k} \psi_{j,k}(x_i) \rvert \, \lvert \kappa_{2^{J-1},H}(x_i,x) \rvert \right)^2 \dd \mu(x)\\
	&\lesssim \sup_{x \in \sphere^d} \, \lVert \kappa_{2^{J-1},H}( \boldsymbol{\cdot},x) \rVert_{L^1(\sphere^d)} \int_{\sphere^d} \frac 1 n \sum_{i=1}^n \lvert \sum_{k=1}^{N_j} \beta_{j,k} \psi_{j,k}(x_i) \rvert^2 \lvert \kappa_{2^{J-1},H}(x_i,x) \rvert \dd \mu(x)\\
	&\lesssim \frac 1 n \sum_{i=1}^n \lvert \sum_{k=1}^{N_j} \beta_{j,k} \psi_{j,k}(x_i) \rvert^2 \, \sup_{i=1,\ldots,n} \lVert \kappa_{2^{J-1},H}(x_i,\boldsymbol{\cdot}) \rVert_{L^1(\sphere^d)}\\
	&\lesssim \left( \frac{2^j}{2^J} \right)^d \left\| \sum_{k=1}^{N_j} \beta_{j,k} \psi_{j,k} \right\| ^2_{L^2(\sphere^d)},
\end{align*}
where we use Lemma~\ref{lemma:Dai06} twice (in each case with $p_0=2$ but first with $p=1$ and then with $p=2$).
Hence, again with the estimate from \cite{Baldi2009a}, p.~3383,
\begin{align}
	\sum_{j > J} &\left\|  \frac{1}{n} \sum_{i=1}^n \left( \sum_{k=1}^{N_j} \beta_{j,k} \psi_{j,k}(x_i) \right) \kappa_{2^{J-1},H}( x_i,\boldsymbol{\cdot}) \right\|_{L^2(\sphere^d)}\notag\\
	&\lesssim \sum_{j > J} \left( \frac{2^j}{2^J} \right)^{d/2} \left\|  \sum_{k=1}^{N_j} \beta_{j,k} \psi_{j,k} \right\| _{L^2(\sphere^d)}\notag\\
	&\lesssim 2^{-Jd/2} \sum_{j > J} 2^{jd/2} \cdot 2^{-j(s-d/r+d/2)}\notag\\
	&= 2^{-Jd/2} \sum_{j > J} 2^{-j(s-d/r)}\notag\\
	&\asymp 2^{-J(s-d/r+d/2)}.\label{eq:Besov:Delta2:2}
\end{align}
Combining the bounds \eqref{eq:Besov:Delta2:1} and \eqref{eq:Besov:Delta2:2}, we obtain, uniformly for $f \in \Theta$,
\begin{equation*}
	\lVert f - \ftilde_J \rVert_{L^2(\sphere^d)} \lesssim 2^{-J(s-d/r+d/2)},
\end{equation*}
which implies that
\begin{align}
	\Delta_2 &\leq 1 - 2\Phi  \left( - \frac{\sqrt n}{2\sigma} \sup_{f \in \Theta} \, \lVert f - \ftilde_J \rVert_{L^2(\sphere^d)} \right)\notag\\[2mm]
	&\lesssim \sigma^{-1} n^{1/2} \sup_{f \in \Theta} \, \lVert f - \ftilde_J \rVert_{L^2(\sphere^d)}\notag\\
	&\lesssim \sigma^{-1} n^{1/2} 2^{-J(s-d/r+d/2)}\label{eq:bound:Delta_2:Besov}.
\end{align}
Combining the bounds \eqref{eq:Besov:Delta1} and \eqref{eq:bound:Delta_2:Besov} finishes the proof of Theorem~\ref{thm:fixed:Besov}.

\subsection{Proof of Theorem~\ref{thm:random:Sobolev}}\label{s:proof:thm:random:Sobolev}

As in the fixed design case, we denote with $\langle \boldsymbol\cdot,\boldsymbol\cdot\rangle_n$ the empirical inner product defined by $$\langle f,g\rangle_n = \frac{1}{n}\sum_{i=1}^n f(X_i)g(X_i)$$
and write $\lVert \,\boldsymbol\cdot\, \rVert_n$ for the associated norm.
$\Pi^n_V$ denotes the orthogonal projection on some subspace $V$ with respect to $\langle \boldsymbol\cdot,\boldsymbol\cdot\rangle_n$.

We first notice that the random design regression experiment $\Rexp_n^d$ given by Equation~\eqref{eq:regression:random:design} is asymptotically equivalent to the experiment $\check \Rexp_n^d$ with \eqref{eq:regression:random:design} replaced by
\begin{equation*}
	\check Z_i = \Pi_{\Pscr^d_L}f(X_i) + \sigma \epsilon_i, \quad i=1,\ldots,n.
\end{equation*}
Indeed, by conditioning on the design $\XX$, we have for $Z=(Z_1,\ldots,Z_n)^\top$ and $\check Z = (\check Z_1,\ldots,\check Z_n)^\top$ the bound
\begin{align*}
	V(\Lc(Z|\XX),\Lc(\check Z | \XX)) &= 1 - 2\Phi \left( -\frac{\sqrt n}{2\sigma} \, \lVert f-\Pi_{\Pscr^d_L} f\rVert_n \right)  \\[2mm]
	&\lesssim \sigma^{-1}n^{1/2} \lVert f-\Pi_{\Pscr^d_L} f\rVert_n.
\end{align*}
Using the conditioning property for $f$-divergences (see \cite{Polyanskiy2024}, p.~117) and Jensen's inequality, we obtain, uniformly for $f \in \Theta$,
\begin{align*}
	V^2(\Lc(Z,\XX),\Lc(\check Z,\XX)) &= (\Eb [V(\Lc(Z|\XX),\Lc(\check Z|\XX))])^2\\[1mm]
	&\leq \Eb [V^2(\Lc(Z|\XX),\Lc(\check Z|\XX))]\\
	&\lesssim \sigma^{-2} \sum_{i=1}^{n} \Eb [ (f(X_i)-\Pi_{\Pscr^d_L} f(X_i))^2 ]\\
	&= \sigma^{-2} n \lVert f - \Pi_{\Pscr^d_L} f \rVert^2_{L^2(\sphere^d)}\\[2mm]
	&\lesssim  \sigma^{-2}n L^{-2s}R^2. 
\end{align*}
Hence,
\begin{equation*}
	\Delta(\Rexp^d_n,\check \Rexp^d_{n}) \lesssim \sigma^{-1} n^{1/2} L^{-s}R.
\end{equation*}
Analogously, we bound the Le Cam distance between the Gaussian white noise experiment $\Gexp_n^d$ and the truncated experiment $\widetilde \Gexp_n^d$ given by
\begin{equation*}
	\Ztilde = \Pi_{\Pscr^d_L} f(x) \dd \mu(x) + \frac{\sigma}{\sqrt n} \dd W(x).
\end{equation*}
More precisely, a direct application of inequality \eqref{eq:app:bound:tv:gwnm} yields that
\begin{align*}
	\Delta(\widetilde\Gexp^d_{n}, \Gexp^d_{n}) &\leq 1 - 2 \Phi \left( - \frac{\sqrt n}{2\sigma}\, \sup_{f \in \Theta}\, \lVert f-\Pi_{\Pscr^d_L} f \rVert_{L^2(\sphere^d)} \right)\\
	&\lesssim \sigma^{-1}n^{1/2}\, \sup_{f \in \Theta}\, \lVert f-\Pi_{\Pscr^d_L} f \rVert_{L^2(\sphere^d)}\\
	&\leq \sigma^{-1}n^{1/2}L^{-s}R.
\end{align*}
After this reduction to two experiments with the same truncated parameter, it remains to find a bound for the Le Cam distance $\Delta(\check \Rexp^d_n,\widetilde \Gexp_n^d)$, that is, we can assume without loss of generality that $f \in \Pscr_L^d$ for the rest of the proof without further reference.

In order to obtain the bound for $\Delta(\check \Rexp^d_n,\widetilde \Gexp_n^d)$, we consider three further intermediate experiments (denoted by $\Iexp_{n,1}^d$, $\Iexp_{n,2}^d$, and $\Iexp_{n,3}^d$ below).
To state these experiments, we need some preparation.
Consider the event
\begin{equation*}
	\Omega_L \defeq \left\lbrace  \frac 1 2 \, \lVert f \rVert_{L^2(\sphere^d)}^2 \leq \lVert f \rVert_{n}^2 \leq  \frac 3 2 \, \lVert f \rVert_{L^2(\sphere^d)}^2 \quad \forall f \in \mathsf \Pscr^d_L \right\rbrace.
\end{equation*}

Similar to the proof of Theorem~\ref{thm:fixed:general}, we denote with $X \in \R^{n \times D}$ the design matrix $(Y_j(X_i))$ where $j=\iota(\ell,m)$ for the enumeration function $\iota$ introduced in Section~\ref{subsec:spherical_harmonics} and $Y_1,\ldots,Y_D$ are a complete set of spherical harmonics up to resolution level $L$, that is, $D=\iota(L,N^d_L)$ and $Y_1,\ldots,Y_D$ are an $L^2(\sphere^d)$-orthonormal basis of $\Pscr^d_L$.
If the matrix $X \in \R^{n \times D}$ has full column rank (this is always the case on $\Omega_L$), we apply the generalized thin $QR$ decomposition (described in detail in Section~\ref{s:QR:decomposition}) to obtain a matrix $Q \in \R^{n \times D}$ with orthonormal columns and an upper triangular block matrix $R$ such that
\begin{equation*}
	X=QR.
\end{equation*}
Here, the choice of the blocks is as in Proposition~\ref{prop:R}, that is, the diagonal blocks are in one-to-one correspondence with the eigenspaces $\Hscr^d_\ell$, $\ell=1,\ldots,L$.
We define functions $Y_1^n,\ldots,Y_D^n$ by
\begin{equation*}
	\begin{pmatrix}
		Y_1^n(x)\\
		\vdots\\
		Y_D^n(x)
	\end{pmatrix} = \sqrt n \, (R^{-1})^\transpose \begin{pmatrix}
		Y_1(x)\\
		\vdots\\
		Y_D(x)
	\end{pmatrix}, \quad x \in \sphere^d.
\end{equation*}
The fact that $(R^{-1})^\transpose$ is a block lower triangular matrix implies that $Y_j^n \in \Pscr^d_\ell$ if $j=\iota(\ell,m)$ for some $\ell \leq L$ and $m \in \{ 1,\ldots,N^d_\ell \}$.
Consider the mapping $T\colon \Pscr^d_L \to \Pscr^d_L$ defined by $TY_j = Y_j^n$ for $j=1,\ldots,D$.
Since the matrix $Q$ in the generalized thin $QR$ decomposition has orthonormal columns, $Y_1^n,\ldots,Y_D^n$ are orthonormal with respect to the empirical scalar product $\langle \boldsymbol\cdot,\boldsymbol\cdot \rangle_n$, and consequently $T\colon (\Pscr^d_L,\lVert \,\boldsymbol\cdot\, \rVert_{L^2(\sphere^d)}) \to (\Pscr^d_L, \lVert \,\boldsymbol\cdot\, \rVert_n)$ is an isometry.
We have
\begin{equation}\label{eq:rel:Tinv:R}
	\langle T^{-1}Y_j,Y_k \rangle_{L^2(\sphere^d)} = \langle Y_j, Y_k^n \rangle_n = n^{-1/2}\, Q_{1:n,k}^\transpose X_{1:n,j} = n^{-1/2} R_{k,j},
\end{equation}
where $R_{k,j}$ is the entry of $R$ in the $k$-th row and $j$-th column.

Set $D_0=\iota(L_0,N^d_{L_0})$ for the intermediate resolution level $L_0 \leq L$ in the statement of Theorem~\ref{thm:random:Sobolev}.
Departing from observations in the random uniform design regression experiment $\check\Rexp_n^d$ with regression function $f \in \Pscr^d_L \cap \Theta$, we define, given the design $\XX$, a first intermediate continuous experiment $\Iexp_{n,1}^d$ with observation $\Ztilde_1$ given by
\begin{align}\label{eq:def:Ztilde_1}
	\Ztilde_1 = &\sum_{j=1}^{D_0} \langle \check Z,Y_j^n \rangle_n Y_j^n + \sum_{j=D_0+1}^{D} \langle \check Z,Y_j^n \rangle_n Y_j + \frac{\sigma}{\sqrt n} \zeta_1,
\end{align}
where $\zeta_1$ is standard Gaussian white noise on the $L^2(\sphere^d)$-orthogonal complement of $\Pscr_{L}^d$ independent of the noise in the lower frequencies.
Note that the observations $\check Z = (\check Z_1,\ldots,\check Z_n)$ in the experiment $\check\Rexp_n^d$ and $\Ztilde_1$ in the experiment $\Iexp_{n,1}^d$ can be transferred into one another provided that $X$ has full column rank.
Obviously, \eqref{eq:def:Ztilde_1} defines $\Ztilde_1$ in terms of $\Zcheck$ under this assumption. 
Vice versa, observations following the same distribution as $\Zcheck$ can be generated from \eqref{eq:def:Ztilde_1} by transforming the high-frequency part by application of the mapping $T$ (that is, $Y_j$ is replaced with $Y_j^n$ for $j=D_0+1,\ldots,D$) and then using an analogous argument as in the proof of Theorem~\ref{thm:fixed:general}.
Defining Markov kernels between the underlying measurable spaces arbitrarily on the null set where $X$ does not have full column rank implies $\Delta(\check\Rexp_n^d,\Iexp_{n,1}^d)=0$.

In addition to $\Iexp_{n,1}^d$, we define two further intermediate experiments $\Iexp_{n,2}^d$ and $\Iexp_{n,3}^d$, which are both defined, conditional on the design $\XX$, by observations $\Ztilde_2$ and $\Ztilde_3$ defined as follows:
\begin{align*}
	\Ztilde_2 &= \sum_{j=1}^{D_0} \langle f,Y_j^n \rangle_n Y_j^n +  \sum_{j=D_0+1}^{D} \langle f,Y_j^n \rangle_n Y_j + \frac{\sigma}{\sqrt n} \zeta_2,\\ 
	\Ztilde_3 &= \sum_{j=1}^{D_0} \langle f,Y_j \rangle_{L^2(\sphere^d)} Y_j + \sum_{j=D_0+1}^{D} \langle f,Y_j^n \rangle_n Y_j + \frac{\sigma}{\sqrt n} \zeta_3.
\end{align*}
Here, both $\zeta_2$ and $\zeta_3$ denote standard Gaussian white noise on $L^2(\sphere^d)$.
In the following, we will bound the Le Cam distance $\Delta(\check \Rexp_n^d, \widetilde \Gexp_n^d)$ by means of the triangle inequality,
\begin{align*}
	\Delta(\check \Rexp_n^d, \widetilde \Gexp_n^d) &\leq \Delta(\check \Rexp_n^d, \Iexp_{n,1}^d) + \Delta(\Iexp_{n,1}^d,\Iexp_{n,2}^d) + \Delta(\Iexp_{n,2}^d,\Iexp_{n,3}^d) + \Delta(\Iexp_{n,3}^d,\widetilde \Gexp_{n}^d)\\
	&=\Delta(\Iexp_{n,1}^d,\Iexp_{n,2}^d) + \Delta(\Iexp_{n,2}^d,\Iexp_{n,3}^d) + \Delta(\Iexp_{n,3}^d,\widetilde \Gexp_{n}^d).
\end{align*}
In order to bound these three terms, it is sufficient to work on the event $\Omega_L$, since by \eqref{app:prob:bound:omega:complement} with $r=2$ we have the bound
\begin{align*}
	\Delta(\Iexp_{n,1}^d,\Iexp_{n,2}^d) &\leq V(\Lc(\Ztilde_1,\XX),\Lc(\Ztilde_2,\XX))\\
	&\leq \Eb [ V(\Lc(\Ztilde_1|\XX),\Lc(\Ztilde_2|\XX)) (\1_{\Omega_L} + \1_{\Omega_L^\complement})]\\
	&\lesssim \Eb [ V(\Lc(\Ztilde_1|\XX),\Lc(\Ztilde_2|\XX)) \1_{\Omega_L}] + \Pb(\Omega_L^\complement)\\
	&\lesssim \Eb [ V(\Lc(\Ztilde_1|\XX),\Lc(\Ztilde_2|\XX)) \1_{\Omega_L}] + n^{-2}
\end{align*}
(of course, the terms $\Delta(\Iexp_{n,2}^d,\Iexp_{n,3}^d)$ and $\Delta(\Iexp_{n,3}^d,\widetilde \Gexp_{n}^d)$ can be treated analogously), and it remains to find appropriate bounds for the terms
\begin{align*}
	&\Eb [ V(\Lc(\Ztilde_1|\XX),\Lc(\Ztilde_{2}|\XX)) \1_{\Omega_L}],\\
	&\Eb [ V(\Lc(\Ztilde_2|\XX),\Lc(\Ztilde_{3}|\XX)) \1_{\Omega_L}],
\end{align*}
and
\begin{equation*}
	\Eb [ V(\Lc(\Ztilde_3|\XX),\Lc(\Ztilde|\XX)) \1_{\Omega_L}],
\end{equation*}
which will finish the proof.
Before we consider these three terms separately, we recall that the transformations used to define $\Ztilde_1$, $\Ztilde_2$, and $\Ztilde_3$ are only well-defined if $X$ has full column rank.
As already mentioned above, this condition holds true on the event $\Omega_L$.
If $X$ does not have full column rank, one can define the Markov kernels that transform between the considered experiments arbitrarily.

\paragraph{Bound on $\Eb [ V(\Lc(\Ztilde_1|\XX),\Lc(\Ztilde_{2}|\XX)) \1_{\Omega_L}]$} 
Following the notation introduced in Section~\ref{s:QR:decomposition}, we denote with $X_{1:n,\underline{1}:\underline{L}_0}$ the submatrix of $X$ consisting only of the first $D_0$ columns.
Note that $\Ztilde_1$ can be written as
\begin{align*}
	\Ztilde_1 &= \sum_{j=1}^{D_0} \langle f,Y_j^n \rangle_n Y_j^n + \sum_{j=1}^{D_0} \beta'_j Y_j+ \sum_{j=D_0+1}^{D} \langle f,Y_j^n \rangle_n Y_j + \frac{\sigma}{\sqrt n} \widetilde\zeta_1, 
\end{align*}
where $\widetilde \zeta_1$ is standard Gaussian white noise on the $L^2(\sphere^d)$-orthogonal complement of $\Pscr^d_{L_0}$ independent of the noise in the low-frequency part.
The mean zero vector $\beta'= (\beta'_1,\ldots,\beta'_{D_0})^\top$ is defined by
\begin{equation*}
	\beta' = (X_{1:n,\underline{1}:\underline{L}_0}^\transpose X_{1:n,\underline{1}:\underline{L}_0})^{-1}X_{1:n,\underline{1}:\underline{L}_0}^\transpose \epsilon
\end{equation*}
for $\epsilon=(\epsilon_1,\ldots,\epsilon_n)^\top \sim \Nc(0,\sigma^2)^{\otimes n}$ and has covariance matrix $$\Cov(\beta') = \sigma^2 (X_{1:n,\underline{1}:\underline{L}_0}^\transpose X_{1:n,\underline{1}:\underline{L}_0})^{-1}.$$
The processes $\Ztilde_1$ and $\Ztilde_2$ have the same mean but different covariance structure, and applying \cite{Devroye2018}, Theorem~1.1, combined with Equation~(2) from the same reference, we obtain that
\begin{align*}
	V(\Lc(\Ztilde_1|\XX),\Lc(\Ztilde_2|\XX))\1_{\Omega_L} &\lesssim \lVert n^{-1} X_{1:n,\underline{1}:\underline{L}_0}^\transpose X_{1:n,\underline{1}:\underline{L}_0} - I_{D_0} \rVert_{\mathrm{F}},
\end{align*}
where $\lVert \,\boldsymbol\cdot\, \rVert_{\mathrm{F}}$ denotes the Frobenius norm of a matrix.
Taking expectations and using the addition formula \eqref{eq:addition:formula}, we obtain that
\begin{align*}
	\Eb [ V^2(\Lc(\Ztilde_1|\XX),\Lc(\Ztilde_{2}|\XX)) \1_{\Omega_L}] &\lesssim \Eb \lVert n^{-1} X_{1:n,\underline{1}:\underline{L}_0}^\transpose X_{1:n,\underline{1}:\underline{L}_0} - I_{D_0} \rVert_{\mathrm{F}}^2\\
	&= \frac{1}{n} \sum_{\ell=0}^{L_0} \sum_{\ell'=0}^{L_0} \sum_{m=1}^{N^d_\ell} \sum_{m'=1}^{N^d_{\ell'}} \Var (Y_{\ell,m}(X_1)Y_{\ell',m'}(X_1))\\
	&\leq \frac{1}{n} \sum_{\ell=0}^{L_0} \sum_{\ell'=0}^{L_0} \sum_{m=1}^{N^d_\ell} \sum_{m'=1}^{N^d_{\ell'}} \int_{\sphere^d} Y_{\ell,m}^2(x) Y_{\ell',m'}^2(x) \dd \mu(x)\\
	&\asymp \frac{1}{n} \left( \sum_{\ell=0}^{L_0} N_\ell^d \right)^2\\
	&\asymp n^{-1} L_0^{2d}.
\end{align*}
Applying Jensen's inequality yields
\begin{equation*}
	\Eb [ V(\Lc(\Ztilde_1|\XX),\Lc(\Ztilde_{2}|\XX)) \1_{\Omega_L}] \lesssim n^{-1/2} L_0^{d}.
\end{equation*}

\paragraph{Bound on $\Eb [ V(\Lc(\Ztilde_2|\XX),\Lc(\Ztilde_{3}|\XX)) \1_{\Omega_L}]$} 

Note that on the event $\Omega_L$ we have that
\begin{equation*}
	\lVert n (X_{1:n,\underline{1}:\underline{L}_0}^\transpose X_{1:n,\underline{1}:\underline{L}_0})^{-1} \rVert_2 \lesssim 1.
\end{equation*}
By combining this bound with Equation~\eqref{eq:app:bound:tv:gwnm}, we have
\begin{align*}
	V(\Lc(\Ztilde_2|\XX),\Lc(\Ztilde_3|\XX)) \1_{\Omega_L} &\lesssim \frac{\sqrt n}{\sigma} \, \lVert \sum_{j=1}^{D_0} \langle f,Y_j^n \rangle_n Y_j^n - \sum_{j=1}^{D_0} \langle f,Y_j \rangle_{L^2(\sphere^d)} Y_j \rVert_{L^2(\sphere^d)} \1_{\Omega_L}\\
	&= \frac{\sqrt n}{\sigma} \, \lVert \sum_{j=1}^{D_0} \langle f,Y_j^n \rangle_n Y_j^n - \Pi_{\Pscr^d_{L_0}}f \rVert_{L^2(\sphere^d)} \1_{\Omega_L}\\
	&= \frac{\sqrt n}{\sigma} \, \lVert (X_{1:n,\underline{1}:\underline{L}_0}^\transpose X_{1:n,\underline{1}:\underline{L}_0})^{-1} X_{1:n,\underline{1}:\underline{L}_0}^\transpose v \rVert \1_{\Omega_L}\\
	&\lesssim \frac{1}{\sigma\sqrt n} \, \lVert X_{1:n,\underline{1}:\underline{L}_0}^\transpose v \rVert,
\end{align*}
where $v=(f(X_1),\ldots,f(X_n))^\transpose-(\Pi^d_{L_0}f(X_1),\ldots,\Pi^d_{L_0}f(X_n))^\transpose$.
Using the addition formula \eqref{eq:addition:formula}, we obtain
\begin{align*}
	\Eb \lVert X_{1:n,\underline{1}:\underline{L}_0}^\transpose v \rVert^2 &= \sum_{\ell=0}^{L_0} \sum_{m=1}^{N^d_\ell} \Eb \lvert \sum_{i=1}^{n} Y_{\ell,m}(X_i) \sum_{\ell' = L_0+1}^L \sum_{m'=1}^{N^d_{\ell'}} \theta_{\ell',m'} Y_{\ell',m'}(X_i) \rvert^2\\
	&= n \sum_{\ell=0}^{L_0} \sum_{m=1}^{N^d_\ell} \Eb \lvert Y_{\ell,m}(X_1) \sum_{\ell' = L_0+1}^L \sum_{m'=1}^{N^d_{\ell'}} \theta_{\ell',m'} Y_{\ell',m'}(X_1) \rvert^2\\
	&= n \Eb \left[ \left( \sum_{\ell=0}^{L_0} \sum_{m=1}^{N^d_\ell} Y_{\ell,m}^2(X_1) \right) \,  \left( \sum_{\ell' = L_0+1}^L \sum_{m'=1}^{N^d_{\ell'}} \theta_{\ell',m'} Y_{\ell',m'}(X_1) \right) ^2 \right]\\
	&\asymp n L_0^d \sum_{\ell' = L_0+1}^L \sum_{m'=1}^{N^d_{\ell'}} \theta_{\ell',m'}^2\\
	&\leq nL_0^{d-2s}R^2,
\end{align*}
from which we conclude that
\begin{align*}
	\Eb [ V(\Lc(\Ztilde_2|\XX),\Lc(\Ztilde_{3}|\XX)) \1_{\Omega_L}] \lesssim \sigma^{-1}  L_0^{d/2-s}R.
\end{align*}

\paragraph{Bound on $\Eb [ V(\Lc(\Ztilde_3|\XX),\Lc(\Ztilde|\XX)) \1_{\Omega_L}]$} 

We have
\begin{align*}
	&V^2(\Lc(\Ztilde_3|\XX),\Lc(\Ztilde|\XX))\\ &\lesssim \sigma^{-2}n \, \lVert \sum_{j=D_0+1}^D (\langle f,Y_j^n \rangle_n - \langle f,Y_j \rangle_{L^2(\sphere^d)}) Y_j \rVert^2_{L^2(\sphere^d)}\\
	&= \sigma^{-2}n \, \lVert \sum_{j=D_0+1}^D \langle T^{-1}f-f,Y_j \rangle_{L^2(\sphere^d)} Y_j \rVert^2_{L^2(\sphere^d)}\\
	&= \sigma^{-2}n \, \langle  (\Pi_{\Pscr_L^d} - \Pi_{\Pscr^d_{L_0}})(T^{-1}f-f),  (\Pi_{\Pscr_L^d} - \Pi_{\Pscr^d_{L_0}})(T^{-1}f-f) \rangle_{L^2(\sphere^d)}.
\end{align*}
By \eqref{eq:rel:Tinv:R}, the matrix corresponding to the mapping $T^{-1}$ (in terms of the spherical harmonics $Y_1,\ldots,Y_D$) is the block upper triangular matrix $n^{-1/2}R$, which implies that
\begin{equation*}
	 (\Pi_{\Pscr^d_L} - \Pi_{\Pscr^d_{L_0}})T^{-1}f= (\Pi_{\Pscr^d_L} - \Pi_{\Pscr^d_{L_0}}) T^{-1}(f-\Pi_{\Pscr^d_{L_0}} f).
\end{equation*}
Hence,
\begin{align*}
 &\langle (\Pi_{\Pscr^d_L} - \Pi_{\Pscr^d_{L_0}}) T^{-1}f, (\Pi_{\Pscr^d_L} - \Pi_{\Pscr^d_{L_0}}) T^{-1}f \rangle_{L^2(\sphere^d)}\\
 \leq~& \langle T^{-1}(f-\Pi_{\Pscr^d_{L_0}} f),T^{-1}(f-\Pi_{\Pscr^d_{L_0}} f) \rangle_{L^2(\sphere^d)}.
\end{align*}
Combining the identity
\begin{equation*}
	\Eb \langle T^{-1}\Pi_{\Hscr^d_\ell}f,T^{-1}\Pi_{\Hscr^d_{\ell'}}f \rangle_{L^2(\sphere^d)}=\langle \Pi_{\Hscr^d_\ell}f,\Pi_{\Hscr^d_{\ell'}}f \rangle_{L^2(\sphere^d)} = \lVert \Pi_{\Hscr^d_{\ell}}f \rVert_{L^2(\sphere^d)}^2 \delta_{\ell,\ell'}
\end{equation*}
for $\ell,\ell' \in \N_0$, $\ell,\ell' \leq L$,
with Proposition~\ref{prop:R}, \ref{it:prop:R:b}, yields that
\begin{align*}
	\Eb V^2(\Lc(\Ztilde_3|\XX),\Lc(\Ztilde|\XX))
	\lesssim \sigma^{-2}n \, \sum_{\ell=L_0+1}^L \Eb  \langle \Pi_{\Hscr^d_\ell}f - T^{-1} \Pi_{\Hscr^d_\ell}f, \Pi_{\Hscr^d_\ell}f \rangle_{L^2(\sphere^d)}.
\end{align*}
We decompose
\begin{align*}
\Eb \langle \Pi_{\Hscr^d_\ell} f - T^{-1}\Pi_{\Hscr^d_\ell} f, \Pi_{\Hscr^d_\ell} f \rangle_{L^2(\sphere^d)}&= \Eb \langle \Pi_{\Hscr^d_\ell} f - T^{-1}\Pi_{\Hscr^d_\ell} f, \Pi_{\Hscr^d_\ell} f \rangle_{L^2(\sphere^d)} \1_{\Omega_L^\complement}\\
	&\phantom{=}+ \Eb \langle \Pi_{\Hscr^d_\ell} f - T^{-1}\Pi_{\Hscr^d_\ell} f, \Pi_{\Hscr^d_\ell} f \rangle_{L^2(\sphere^d)} \1_{\Omega_L}.
\end{align*}
The Cauchy--Schwarz inequality yields
\begin{align*}
	&\Eb \langle \Pi_{\Hscr^d_\ell} f - T^{-1}\Pi_{\Hscr^d_\ell} f, \Pi_{\Hscr^d_\ell} f \rangle_{L^2(\sphere^d)} \1_{\Omega_L^\complement}\\
	&~\leq (\Eb \langle \Pi_{\Hscr^d_\ell} f - T^{-1}\Pi_{\Hscr^d_\ell} f, \Pi_{\Hscr^d_\ell} f \rangle_{L^2(\sphere^d)}^2)^{1/2} \cdot (\Pb(\Omega_L^\complement))^{1/2}.
\end{align*}
Combining the estimate $(a\pm b)^2 \leq 2a^2+2b^2$, the Cauchy--Schwarz inequality, and the fact that $T$ is an isometry yields that
\begin{align*}
	&\Eb \langle \Pi_{\Hscr^d_\ell} f - T^{-1}\Pi_{\Hscr^d_\ell} f, \Pi_{\Hscr^d_\ell} f \rangle_{L^2(\sphere^d)}^2\\
	\leq~& 2 \Eb [ \lVert \Pi_{\Hscr^d_\ell} f\rVert^4_{L^2(\sphere^d)} + \langle T^{-1}\Pi_{\Hscr^d_\ell} f,\Pi_{\Hscr^d_\ell} f \rangle^2_{L^2(\sphere^d)} ]\\
	\leq~& 2\lVert \Pi_{\Hscr^d_\ell} f\rVert^4_{L^2(\sphere^d)} + 2 \lVert \Pi_{\Hscr^d_\ell} f \rVert_{L^2(\sphere^d)}^2 \Eb [\lVert T^{-1}\Pi_{\Hscr^d_\ell} f \rVert^2_{L^2(\sphere^d)} ]\\
	\leq~& 4\lVert \Pi_{\Hscr^d_\ell} f\rVert^4_{L^2(\sphere^d)}.
\end{align*}
Hence, by combining the last estimate with Equation~\eqref{app:prob:bound:omega:complement} with $r=2$, we get
\begin{equation*}
	\Eb \langle \Pi_{\Hscr^d_\ell} f - T^{-1}\Pi_{\Hscr^d_\ell} f, \Pi_{\Hscr^d_\ell} f \rangle_{L^2(\sphere^d)} \1_{\Omega_L^\complement} \lesssim n^{-1} \lVert \Pi_{\Hscr^d_\ell} f\rVert^2_{L^2(\sphere^d)}. 
\end{equation*}
Identifying the projection $\Pi_{\Hscr^d_\ell} f$ with the coefficient vector
\begin{equation*}
	\theta_{\ell,\boldsymbol \cdot} \defeq (\theta_{\ell,1},\ldots,\theta_{\ell,N_\ell^d})^\top \in \R^{N^d_\ell},
\end{equation*}
we can write
\begin{align*}
	\langle \Pi_{\Hscr^d_\ell} f - T^{-1}\Pi_{\Hscr^d_\ell} f, \Pi_{\Hscr^d_\ell} f \rangle_{L^2(\sphere^d)}
	&~=~ \theta_{\ell,\boldsymbol \cdot}^\transpose \theta_{\ell,\boldsymbol \cdot} - n^{-1/2} \theta_{\ell,\boldsymbol \cdot}^\transpose R_{\underline\ell,\underline\ell} \theta_{\ell,\boldsymbol \cdot}\\
	&~=~ \theta_{\ell,\boldsymbol \cdot}^\transpose (I_{N^d_\ell}- n^{-1/2}R_{\underline\ell,\underline\ell}) \theta_{\ell,\boldsymbol \cdot}\\
	&~=~ \theta_{\ell,\boldsymbol \cdot}^\transpose ( I_{N^d_\ell}- M_\ell^{1/2} )  \theta_{\ell,\boldsymbol \cdot},
\end{align*}
where the matrix $M_\ell \in \R^{N^d_\ell \times N^d_\ell}$ is defined by
\begin{equation*}
	M_\ell \defeq n^{-1} X_{1:n,\underline{\ell}}^\top  X_{1:n,\underline{\ell}} - n^{-1} X_{1:n,\underline{\ell}}^\top Q_{1:n,\underline{1}:\underline{\ell-1}}Q_{1:n,\underline{1}:\underline{\ell-1}}^\top X_{1:n,\underline{\ell}}.
\end{equation*}
On $\Omega_L$ we have
\begin{equation*}
	\lVert n^{-1} X_{1:n,\underline{\ell}}^\top  X_{1:n,\underline{\ell}} - I_{N^d_\ell} \rVert_2 \leq \frac{1}{2},
\end{equation*}
which implies that the spectrum of $M_\ell$ is contained in the interval $(0,3/2]$.
Consequently, the spectrum of $M_\ell- I_{N^d_\ell}$ is contained in $(-1,1/2]$ on $\Omega_L$.
Using the Taylor series expansion of the matrix function $A \mapsto (I+A)^{1/2}$ (which converges for $A$ with spectrum contained in $(-1,1)$; see \cite{Higham2008}, Theorem~4.7), we obtain that 
\begin{align*}
	( I_{N^d_\ell} - M_\ell^{1/2} ) \1_{\Omega_L}
	&= \left[  I_{N^d_\ell} - \sum_{k=0}^\infty \binom{1/2}{k} (M_\ell-I_{N^d_\ell})^k \right]  \1_{\Omega_L}\\
	&= \sum_{k=1}^\infty \binom{1/2}{k} (M_\ell-I_{N^d_\ell})^k  \1_{\Omega_L}.
\end{align*}
By taking expectations, we obtain
\begin{align*}
	&\Eb \langle \Pi_{\Hscr^d_\ell} f - T^{-1}\Pi_{\Hscr^d_\ell} f, \Pi_{\Hscr^d_\ell} f \rangle_{L^2(\sphere^d)} \1_{\Omega_L}= \frac{1}{2} \theta_{\ell,\boldsymbol \cdot}^\transpose \Eb [M_\ell-I_{N^d_\ell}] \theta_{\ell,\boldsymbol \cdot}\\
	&\quad- \frac{1}{2}  \Eb [\theta_{\ell,\boldsymbol \cdot}^\transpose(M_\ell-I_{N^d_\ell}) \theta_{\ell,\boldsymbol \cdot} \1_{\Omega_L^\complement}] + \Eb \theta_{\ell,\boldsymbol \cdot}^\transpose \sum_{k=2}^\infty \binom{1/2}{k} (M_\ell-I_{N^d_\ell})^k \theta_{\ell,\boldsymbol \cdot} \1_{\Omega_L}.
\end{align*}
First, since $\Eb X_{1:n,\underline{\ell}}^\top  X_{1:n,\underline{\ell}} = n I_{N^d_\ell}$, we obtain
\begin{equation*}
	\frac{1}{2} \theta_{\ell,\boldsymbol \cdot}^\transpose \Eb [M_\ell-I_{N^d_\ell}] \theta_{\ell,\boldsymbol \cdot} \leq 0.
\end{equation*}
Second, by \eqref{app:prob:bound:omega:complement} for $r=2$, Proposition~\ref{prop:reverse_Hoelder}, and Proposition~\ref{prop:estimate:norm:emp:proj}, \ref{it:further:b}, 
\begin{align*}
	\Eb &[\lvert\theta_{\ell,\boldsymbol \cdot}^\transpose(M_\ell-I_{N^d_\ell}) \theta_{\ell,\boldsymbol \cdot} \1_{\Omega_L^\complement}\rvert]\\
	&\leq (\Eb [(\theta_{\ell,\boldsymbol \cdot}^\transpose(M_\ell-I_{N^d_\ell}) \theta_{\ell,\boldsymbol \cdot})^2])^{1/2} (\Pb(\Omega_L^\complement))^{1/2}\\
	&\lesssim (\Var(\lVert \Pi_{\Hscr_\ell^d} f \rVert_n^2) + \Eb \lVert \Pi^n_{\Pscr^d_{\ell-1}} \Pi_{\Hscr_\ell^d} f \rVert_n^4)^{1/2}   n^{-1}\\
	&\leq (n^{-1}\Eb((\Pi_{\Hscr_\ell^d} f(X_1))^4) + \Eb \lVert \Pi^n_{\Pscr^d_{\ell-1}} \Pi_{\Hscr_\ell^d} f \rVert_n^4)^{1/2}n^{-1}\\
	&\lesssim (n^{-1}\lVert \Pi_{\Hscr_\ell^d} f \rVert_{L^2(\sphere^d)}^4 \ell^{s(d)} + \lVert \Pi_{\Hscr_\ell^d} f \rVert_{L^2(\sphere^d)}^4 (\ell^d n^{-1}+\ell^{2d-2}n^{-2}))^{1/2} n^{-1},
\end{align*}
where, according to Proposition~\ref{prop:reverse_Hoelder}, we can take
\begin{equation*}
	s(d) = \begin{cases}
		(d-1)/2, & \text{if } d \in \{ 2,3 \},\\
		d-2, & \text{if } d \geq 4. 
	\end{cases}
\end{equation*}
Hence, for $\ell \leq L$,
\begin{align*}
\Eb [\lvert\theta_{\ell,\boldsymbol \cdot}^\transpose(M_\ell-I_{N^d_\ell}) \theta_{\ell,\boldsymbol \cdot} \1_{\Omega_L^\complement}\rvert] &\lesssim (n^{-1/2} \ell^{s(d)/2} + \ell^{d/2}n^{-1/2} + \ell^{d-1}n^{-1}) \lVert \Pi_{\Hscr_\ell^d} f \rVert_{L^2(\sphere^d)}^2 n^{-1}\\
	&\lesssim n^{-3/2}\ell^{d/2} \lVert \Pi_{\Hscr_\ell^d} f \rVert_{L^2(\sphere^d)}^2.
\end{align*}
Third, the bound
\begin{equation*}
	\bigg\lvert\binom{1/2}{k}\bigg\rvert \lesssim k^{-3/2}
\end{equation*}
yields
\begin{equation*}
	\bigg\lVert\sum_{k=2}^\infty \binom{1/2}{k} (M_\ell-I_{N^d_\ell})^{k-2} \1_{\Omega_L}\bigg\rVert_2 \lesssim 1,
\end{equation*}
from which we obtain that
\begin{align*}
	&\theta_{\ell,\boldsymbol \cdot}^\transpose \sum_{k=2}^\infty \binom{1/2}{k} (M_\ell-I_{N^d_\ell})^k \theta_{\ell,\boldsymbol \cdot} \1_{\Omega_L}\\
	=~& \theta_{\ell,\boldsymbol \cdot}^\transpose (M_\ell-I_{N^d_\ell}) \sum_{k=2}^\infty \binom{1/2}{k} (M_\ell-I_{N^d_\ell})^{k-2} (M_\ell-I_{N^d_\ell})\theta_{\ell,\boldsymbol \cdot} \1_{\Omega_L}\\
	\lesssim~& \theta_{\ell,\boldsymbol \cdot}^\transpose (M_\ell-I_{N^d_\ell})^2 \theta_{\ell,\boldsymbol \cdot} \1_{\Omega_L}\\
	=~& \lVert (M_\ell-I_{N^d_\ell})\theta_{\ell,\boldsymbol \cdot} \rVert^2 \1_{\Omega_L}\\
	\lesssim~& \lVert (n^{-1}X_{1:n,\underline{\ell}}^\transpose X_{1:n,\underline{\ell}}-I_{N^d_\ell})\theta_{\ell,\boldsymbol \cdot} \rVert^2+ \lVert n^{-1}X_{1:n,\underline{\ell}}^\transpose Q_{1:n,\underline{1}:\underline{\ell-1}} Q_{1:n,\underline{1}:\underline{\ell-1}}^\transpose X_{1:n,\underline{\ell}} \theta_{\ell,\boldsymbol \cdot} \rVert^2 \1_{\Omega_L}.
\end{align*}
We have
\begin{align*}
	&\Eb \lVert (n^{-1}X_{1:n,\underline{\ell}}^\transpose X_{1:n,\underline{\ell}}-I_{N^d_\ell})\theta_{\ell,\boldsymbol \cdot} \rVert^2\\
	=&~ \sum_{i=1}^{N^d_\ell} \Var \left( \sum_{j=1}^{N_\ell^d} \frac{1}{n} \sum_{s=1}^{n} Y_{\ell,i}(X_s) Y_{\ell,j}(X_s) \theta_{\ell,j} \right) \\
	=&~ n^{-2}\sum_{i=1}^{N^d_\ell} \Eb \left[ \left( \sum_{j=1}^{N_\ell^d} \sum_{s=1}^{n} Y_{\ell,i}(X_s) Y_{\ell,j}(X_s) \theta_{\ell,j} \right)^2 \right] - \sum_{i=1}^{N^d_\ell} \theta_{\ell,i}^2\\
	=&~n^{-2} \sum_{i=1}^{N^d_\ell} \sum_{j,k=1}^{N_\ell^d} \theta_{\ell,j}\theta_{\ell,k}\sum_{s,t=1}^{n} \Eb Y_{\ell,i}(X_s) Y_{\ell,j}(X_s)  Y_{\ell,i}(X_t) Y_{\ell,k}(X_t) - \sum_{i=1}^{N^d_\ell} \theta_{\ell,i}^2\\
	=&~ n^{-2} \sum_{i=1}^{N^d_\ell} \sum_{j,k=1}^{N_\ell^d} \theta_{\ell,j}\theta_{\ell,k}\sum_{s=1}^{n} \Eb Y^2_{\ell,i}(X_s) Y_{\ell,j}(X_s) Y_{\ell,k}(X_s)\\
	& + n^{-2} \sum_{i=1}^{N^d_\ell} \sum_{j,k=1}^{N_\ell^d} \theta_{\ell,j}\theta_{\ell,k}\sum_{\substack{s,t=1\\s\neq t}}^{n} \Eb Y_{\ell,i}(X_s) Y_{\ell,j}(X_s) \Eb Y_{\ell,i}(X_t) Y_{\ell,k}(X_t)- \sum_{i=1}^{N^d_\ell} \theta_{\ell,i}^2\\
	=&~ n^{-2}N^d_\ell \sum_{j,k=1}^{N_\ell^d} \theta_{\ell,j}\theta_{\ell,k}\sum_{s=1}^{n} \Eb Y_{\ell,j}(X_s) Y_{\ell,k}(X_s) + n^{-2} \sum_{j=1}^{N_\ell^d} \theta_{\ell,j}^2 \sum_{\substack{s,t=1\\s\neq t}}^{n} 1- \sum_{i=1}^{N^d_\ell} \theta_{\ell,i}^2\\
	=&~ n^{-1}N^d_\ell \sum_{j=1}^{N_\ell^d} \theta_{\ell,j}^2 + n^{-1}(n-1) \sum_{j=1}^{N_\ell^d} \theta_{\ell,j}^2 - \sum_{i=1}^{N^d_\ell} \theta_{\ell,i}^2\\
	\leq&~ n^{-1}N^d_\ell \lVert \Pi_{\Hscr_\ell^d} f \rVert_{L^2(\sphere^d)}^2.
\end{align*}
Note that the non-zero eigenvalues of $X_{1:n,\underline{\ell}}X_{1:n,\underline{\ell}}^\transpose$ coincide with those of $X_{1:n,\underline{\ell}}^\transpose X_{1:n,\underline{\ell}}$ and that on the event $\Omega_L$ the positive eigenvalues of $X_{1:n,\underline{\ell}}^\transpose X_{1:n,\underline{\ell}}$ are bounded by $3n/2$.
It follows that
\begin{align*}
	\Eb& \lVert n^{-1}X_{1:n,\underline{\ell}}^\transpose Q_{1:n,\underline{1}:\underline{\ell-1}} Q_{1:n,\underline{1}:\underline{\ell-1}}^\transpose X_{1:n,\underline{\ell}} \theta_{\ell,\boldsymbol \cdot} \rVert^2 \,\1_{\Omega_L}\\
	&\leq \frac{3}{2} \, n^{-1} \Eb \lVert Q_{1:n,\underline{1}:\underline{\ell-1}}Q_{1:n,\underline{1}:\underline{\ell-1}}^\transpose X_{1:n,\underline{\ell}} \theta_{\ell,\boldsymbol \cdot} \rVert^2 \,\1_{\Omega_L}\\
	&= \frac{3}{2} \, \Eb \lVert \Pi^n_{\Pscr^d_{\ell-1}} \Pi_{\Hscr_\ell^d} f \rVert_n^2 \,\1_{\Omega_L}.
\end{align*}
By Proposition~\ref{prop:estimate:norm:emp:proj}, \ref{it:further:a}, we thus obtain for $\ell \leq L$ that
\begin{align*}
	\Eb \lVert n^{-1}X_{1:n,\underline{\ell}}^\transpose Q_{1:n,\underline{1}:\underline{\ell-1}} Q_{1:n,\underline{1}:\underline{\ell-1}}^\transpose X_{1:n,\underline{\ell}} \theta_{\ell,\boldsymbol \cdot} \rVert^2 \1_{\Omega_L} &\lesssim \lVert \Pi_{\Hscr_\ell^d} f \rVert^2_{L^2(\sphere^d)} \ell^d n^{-1}.
\end{align*}
By combining the obtained estimates, we have
\begin{align*}
	\Eb [ V^2(\Lc(\Ztilde_3|\XX),\Lc(\Ztilde|\XX)) \1_{\Omega_L}]&\lesssim \sigma^{-2}\sum_{\ell=L_0+1}^L \ell^d \lVert \Pi_{\Hscr_\ell^d} f \rVert^2_{L^2(\sphere^d)}\\
	&\lesssim \sigma^{-2}L_0^{-2s+d}R^2.
\end{align*}
Thus
\begin{equation*}\Eb [ V(\Lc(\Ztilde_3|\XX),\Lc(\Ztilde|\XX)) \1_{\Omega_L}] \lesssim \sigma^{-1} L_0^{-s+d/2}R,
\end{equation*}
finishing the proof of the theorem.

\subsection{Proof of Theorem~\ref{thm:non-equivalence}}

For the experiments $\Fexp_n^d$ and $\Gexp_n^d$, we reduce the common parameter set $\Theta$ to $\Theta_n' = \{ f_{n,1},f_{n,2} \}$ for suitably defined $f_{n,1}, f_{n,2} \in \Theta$.
For any $n \in \N$ and both $\Theta=H^s_2(\sphere^d,R)$ or $\Theta=B^s_{r,q}(\sphere^d,R)$, we take $f_{n,1} \equiv 0$ (which trivially belongs to both of the smoothness classes $\Theta$ considered).
In order to define $f_{n,2}$, we first choose an integer $L$ such that $2n \leq \dim \Pscr^d_L \leq Cn$ for some suitable constant $C>2$.
Consider the linear subspace
\begin{equation*}
	W_n \defeq \{ f \in \Pscr_L^d : f(x) = 0 \quad \forall x \in \XX \}.
\end{equation*}
We have
\begin{equation*}
	\dim W_n \geq \dim \Pscr^d_L - n \geq \frac{\dim \Pscr_L^d}{2}.
\end{equation*}
Now, by \cite{Dai2013b}, Proposition~3.5, there exists a function $h_n \in W_n$ such that
\begin{equation}\label{eq:h_n:norm:all:p}
	\lVert h_n \rVert_{L^p(\sphere^d)} \asymp 1
\end{equation}
for all $p \in [1,\infty]$.
Based on these preliminaries, we now construct $f_{n,2} \in \Theta$ such that
\begin{equation}\label{eq:prop:f_2n}
	\lVert f_{n,2}\rVert_{L^2(\sphere^d)} \asymp \begin{cases}
		n^{-s/d}, & \text{if } \Theta=H_2^s(\sphere^d,R),\\
		n^{-s/d}, & \text{if } \Theta=B^s_{r,q}(\sphere^d,R) \text{ with } r \geq 2,\\
		n^{-s/d+1/r-1/2}, & \text{if } \Theta=B^s_{r,q}(\sphere^d,R) \text{ with } r < 2.
	\end{cases}
\end{equation}

\paragraph{Case $\Theta=H_2^s(\sphere^d,R)$}
We make the ansatz
\begin{equation*}
	f_{n,2} = cn^{-s/d} h_n
\end{equation*}
with a constant $c>0$ independent of $n$ to be specified now.
Since $f_{n,2} \in W_n \subseteq \Pscr_L^d$, we have by the choice of $L$ and \eqref{eq:h_n:norm:all:p} that
\begin{align*}
	\lVert f_{n,2} \rVert_{H_2^s(\sphere^d)}^ 2  &\lesssim L^{2s} \lVert f_{n,2} \rVert_{L^2(\sphere^d)}^2\asymp c^2L^{2s}n^{-2s/d} \lVert h_n \rVert_{L^2(\sphere^d)}^2\asymp 1,
\end{align*}
showing that $f_{n,2} \in H_2^s(\sphere^d,R)$ for $c$ sufficiently small.
Moreover,
\begin{equation*}
	\lVert f_{n,2} \rVert_{L^2(\sphere^d)} = cn^{-s/d} \lVert h_{n} \rVert_{L^2(\sphere^d)} \asymp n^{-s/d}.
\end{equation*}

\paragraph{Case $\Theta=B^s_{r,q}(\sphere^d,R)$ with $r \geq 2$}

As in the previous case, we put
\begin{equation*}
	f_{n,2}=cn^{-s/d}h_n
\end{equation*}
with a constant $c>0$ independent of $n$ to be chosen appropriately.
Let $m \in \N$ be such that $2^{m-1} \leq L < 2^m$.
Recall the definition of $E_L(f,r)$ given in \eqref{def:E_L}.
Note that $E_{2^j}(f_{n,2},r) \leq \lVert f_{n,2} \rVert_{L^r(\sphere^d)}$ for any $j \in \N_0$ and $E_{2^j}(f_{n,2},r) =0$ for $j \geq m$.
Then, using the equivalence \eqref{eq:equivalent:Besov:norm} of norms, we obtain
\begin{align*}
	\lVert f_{n,2} \rVert_{B^s_{r,q}(\sphere^d)} &\asymp \lVert f_{n,2} \rVert_{L^r(\sphere^d)} + \left( \sum_{j=0}^{m} ( 2^{js} E_{2^j}(f_{n,2},r) )^q \right)^{\frac 1 q}\\[2mm]
	&\lesssim \left( 1 + \left( \sum_{j=0}^m 2^{jsq}  \right)^{\frac 1 q} \right) \lVert f_{n,2} \rVert_{L^r(\sphere^d)}\\[1mm]
	&\lesssim 2^{ms} \lVert f_{n,2} \rVert_{L^r(\sphere ^d)}\\[1mm]
	&= c2^{ms}n^{-s/d} \lVert h_{n} \rVert_{L^r(\sphere ^d)}\\[1mm]
	&\asymp 1.
\end{align*}
Hence, $f_{n,2} \in B^s_{r,q}(\sphere^d,R)$ provided that $c$ is chosen sufficiently small.
Moreover, as in the previous case,
\begin{equation*}
	\lVert f_{n,2} \rVert_{L^2(\sphere^d)} = cn^{-s/d} \lVert h_{n} \rVert_{L^2(\sphere^d)} \asymp n^{-s/d}.
\end{equation*}

\paragraph{Case $\Theta=B^s_{r,q}(\sphere^d,R)$ with $r < 2$}
Let $x_n^\ast \in \sphere^d$ be such that $\lVert h_n \rVert_{L^\infty(\sphere^d)} = \lvert h_n(x_n^\ast) \rvert$, and put
\begin{equation*}
	f_{n,2} = c n^{-s/d+1/r-1} h_n \kappa_{L,H}(\,\boldsymbol\cdot\,, x_n^\ast)
\end{equation*}
with $\kappa_{L,H}$ as defined by \eqref{eq:def:kappa} and \eqref{eq:def:H}.
As in the previous cases, the constant $c>0$ has to be chosen sufficiently small and independent of $n$.
On the one hand,
\begin{equation*}
	\kappa_{L,H}(x_n^\ast,x_n^\ast) \geq \sum_{\ell=0}^{L} N^d_\ell \asymp L^{d},
\end{equation*}
on the other hand,
\begin{equation*}
	\kappa_{L,H}(x_n^\ast,x_n^\ast) \leq \sum_{\ell=0}^{2L} N^d_\ell \asymp L^{d},
\end{equation*}
and hence $\kappa_{L,H}(x_n^\ast,x_n^\ast) \asymp L^d$.
We have
\begin{equation}\label{eq:h_kappa:sup}
	\lVert h_n \kappa_{L,H}(\,\boldsymbol \cdot\,,x_n^\ast) \rVert_{L^\infty(\sphere^d)} = \lvert h_n(x_n^\ast) \rvert \cdot \lvert \kappa_{L,H}(x_n^\ast, x_n^\ast) \rvert \asymp L^d.
\end{equation}
By \cite{Narcowich2006a}, Proposition~2.5, we obtain that, for $1 \leq p < \infty$,
\begin{align}
	\lVert h_n \kappa_{L,H}(\,\boldsymbol \cdot\,,x_n^\ast) \rVert_{L^p(\sphere^d)} &\lesssim L^{d(1-1/p)} \lVert h_n \kappa_{L,H}(\,\boldsymbol \cdot\,,x_n^\ast) \rVert_{L^1(\sphere^d)}\notag\\
	&\lesssim L^{d(1-1/p)} \lVert h_n \rVert_{L^\infty(\sphere^d)} \, \lVert \kappa_{L,H}(\,\boldsymbol \cdot\,,x_n^\ast) \rVert_{L^1(\sphere^d)}\notag\\
	&\lesssim L^{d(1-1/p)},\label{eq:h_kappa:p:upper}
\end{align}
where we used \eqref{eq:h_n:norm:all:p} and \cite{Wang2017a}, Theorem~3.3, in the last estimate.
Using \eqref{eq:h_kappa:sup} and \cite{Narcowich2006a}, Proposition~2.5, again, we have
\begin{equation*}
	L^d \asymp \lVert h_n \kappa_{L,H}(\,\boldsymbol \cdot\,,x_n^\ast) \rVert_{L^\infty(\sphere^d)} \lesssim L^{d/p} \lVert h_n \kappa_{L,H}(\,\boldsymbol \cdot\,,x_n^\ast) \rVert_{L^p(\sphere^d)}. 
\end{equation*}
Combining the last estimate with \eqref{eq:h_kappa:p:upper}, we obtain
\begin{equation}\label{eq:h_kappa:p:asymp}
	\lVert h_n \kappa_{L,H}(\,\boldsymbol \cdot\,,x_n^\ast) \rVert_{L^p(\sphere^d)} \asymp L^{d(1-1/p)}
\end{equation}
also for $1 \leq p < \infty$.
With $m$ being chosen as in the previous case, we obtain using \eqref{eq:h_kappa:p:asymp} that
\begin{align*}
	\lVert f_{n,2} \rVert_{B^s_{r,q}(\sphere^d)} &\leq \left( 1 + \left( \sum_{j=0}^{m+2} 2^{jsq} \right)^{1/q} \right) \lVert f_{n,2} \rVert_{L^r(\sphere^d)}\\
	&\lesssim 2^{ms} n^{-s/d+1/r-1} \lVert h_n \kappa_{L,H}(\,\boldsymbol \cdot\,,x_n^\ast) \rVert_{L^r(\sphere^d)}\\
	&\asymp 2^{ms} n^{-s/d+1/r-1} L^{d(1-1/r)}\\
	&\asymp 1,
\end{align*}
hence $f_{n,2} \in B^s_{r,q}(\sphere^d,R)$ for $c$ sufficiently small.
Moreover, from \eqref{eq:h_kappa:p:asymp} we also obtain
\begin{equation*}
	\lVert f_{n,2} \rVert_{L^2(\sphere^d)} \asymp n^{-s/d+1/r-1/2},
\end{equation*}
which is \eqref{eq:prop:f_2n} for $\Theta = B^s_{r,q}(\sphere^d,R)$ and $r<2$.

\bigskip

Having established \eqref{eq:prop:f_2n} in all three cases of interest, we can now derive \eqref{eq:bound:non-equivalence}.
By the very definition of $W_n$, we have $f_{n,2}\lvert_{\XX} ~ \equiv 0$, which trivially implies that
\begin{equation}\label{eq:proof:non-equivalence:1}
	V(\Pb^{\Fexp_n^d}_{f_{n,1}}, \Pb^{\Fexp_n^d}_{f_{n,2}})=0
\end{equation}
for any $n$.
Combining \eqref{eq:prop:f_2n} with \eqref{eq:app:bound:tv:gwnm}, however, shows that
\begin{equation}\label{eq:proof:non-equivalence:2}
	V(\Pb^{\Gexp_n^d}_{f_{n,1}}, \Pb^{\Gexp_n^d}_{f_{n,2}}) \gtrsim 1
\end{equation}
for all the cases considered.
\cite{Ray2019}, Lemma~1, states that
\begin{equation}\label{eq:deficiency:lower:bound}
	\delta(\Fexp_n^d,\Gexp_n^d) \geq \frac{1}{2} (V(\Pb^{\Gexp_n^d}_{f_{n,1}},\Pb^{\Gexp_n^d}_{f_{n,2}}) - V(\Pb^{\Fexp_n^d}_{f_{n,1}}, \Pb^{\Fexp_n^d}_{f_{n,2}})).
\end{equation}
Plugging \eqref{eq:proof:non-equivalence:1} and \eqref{eq:proof:non-equivalence:2} into \eqref{eq:deficiency:lower:bound} implies that $\delta(\Fexp_n^d,\Gexp_n^d) \gtrsim 1$ and hence the asymptotic non-equivalence of $\Fexp_n^d$ and $\Gexp_n^d$.

\subsection{Auxiliary results}

\subsubsection{A consequence of the Marcinkiewicz--Zygmund condition}

Our asymptotic equivalence results for specific function classes stated in Sections~\ref{sec:AE:Sobolev} and \ref{sec:AE:Besov} rely on recent contributions concerning numerical integration on spheres, for instance, from \cite{Lu2023}.
The main assumption in \cite{Lu2023} is the validity of a certain Marcinkiewicz--Zygmund condition on the cubature points (see \cite{Lu2023}, Equation~(1.3)).
The following technical lemma, which is used in the proofs of Theorems~\ref{thm:fixed:Sobolev} and \ref{thm:fixed:Besov}, is essentially based on \cite{Dai2006}, Theorem~2.1.
Our formulation is slightly closer to the version stated in \cite{Lu2023}, Lemma~3.1.
The condition \eqref{eq:cond:lemma:Dai06} in the lemma is always satisfied for equal-weight spherical $t$-designs when taking $p_0=2$, $t \geq 2L_0$, $C_0=1$ (in fact, in this case even equality holds in \eqref{eq:cond:lemma:Dai06}).
The assertion of Lemma~\ref{lemma:Dai06} then allows us to bound the value of the cubature formula (which is exact on $\Pscr^d_{L_0}$) from above by the target integral also for truncated spherical harmonic expansions where the cubature is not exact anymore (the more coefficients are involved, the less accurate the bound becomes).
This kind of bound is used in the proofs of Theorems~\ref{thm:fixed:Sobolev} and \ref{thm:fixed:Besov}, respectively, to control the empirical norms appearing in the term $\Delta_1$ of Theorem~\ref{thm:fixed:general}.

\begin{lemma}\label{lemma:Dai06}
	Let $\XX = \{ x_1,\ldots,x_n \} \subset \sphere^d$ and $\omega_1,\ldots,\omega_n > 0$ be weights satisfying, for some $L_0 \in \N_0$ and $p_0 \in (0,\infty)$,
	\begin{equation}\label{eq:cond:lemma:Dai06}
		\sum_{i=1}^n \omega_i \lvert f(x_i) \rvert^{p_0} \leq C_0 \int_{\sphere^d} \lvert f(x) \rvert^{p_0} \dd \mu (x), \quad f \in \mathscr P^d_{L_0},
	\end{equation}
	where $C_0 > 0$ is a numerical constant.
	Then, for any $p \in (0,\infty)$ and any integer $L \geq L_0$,
	\begin{equation*}
		\sum_{i=1}^n \omega_i \lvert f(x_i) \rvert^{p} \leq CC_0 \left( \frac{L}{L_0} \right)^d \int_{\sphere^d} \lvert f(x) \rvert^{p} \dd \mu (x), \quad f \in \mathscr P^d_{L},
	\end{equation*}
	with a numerical constant $C=C(d,p)>0$ depending only on $d$ and $p$.
\end{lemma}

\subsubsection{Bound for $\Pb(\Omega_L^\complement)$}\label{subsub:bound:Omegac}

The defining property of the event $\Omega_L$,
\begin{equation*}
	\frac 1 2 \, \lVert f \rVert_{L^2(\sphere^d)}^2 \leq \lVert f \rVert_{n}^2 \leq \frac 3 2 \, \lVert f \rVert_{L^2(\sphere^d)}^2, \quad f \in \Pscr^d_L,
\end{equation*}
is equivalent to
\begin{equation*}
	\lVert n^{-1}X^\top X - I \rVert_2 \leq \frac{1}{2},
\end{equation*}
where $X$ denotes the design matrix defined in the proof of Theorem~\ref{thm:random:Sobolev}.
Applying Theorem~1 from \cite{Cohen2013} (taking into account the corrected numerical constants from \cite{Cohen2018}) with $\delta = 1/2$ (leading to the choice $c_{1/2}=(3 \ln(1.5) - 1)/2 > 0$ in the statement of that theorem) and choosing the \emph{maximal} $L$ such that for $\kappa_r=(3\ln(1.5)-1)/(2+2r)$
\begin{equation*}
	\dim \Pscr^d_L \leq \kappa_r \, \frac{n}{\ln (n)},
\end{equation*}
we obtain the estimate
\begin{equation}\label{app:prob:bound:omega:complement}
	\Pb(\Omega_L^\complement) \leq 2n^{-r}.
\end{equation}
In Theorem~\ref{thm:random:Sobolev} and its proof, we choose $r=2$, which leads to $\kappa_2 \approx 0.036$.

\subsubsection{Generalized thin $QR$ decomposition}\label{s:QR:decomposition}

For the proof of Theorem~\ref{thm:random:Sobolev}, we need a generalization of the classical thin $QR$ decomposition of a rectangular matrix with full column rank (see \cite{Golub2013}, Theorem~5.2.3).
Since we are not aware of a reference for this kind of generalization, we state it here in full detail (a proof can be obtained by an obvious adjustment of the one in the classical case).
Recall that a square matrix $A \in \R^{m \times m}$ is called \emph{positive definite} if the following two conditions hold:
\begin{enumerate}[label=(\roman*)]
	\item $A$ is symmetric, that is, $A=A^\transpose$,
	\item $x^\transpose A x > 0$ for all $x \in \R^m$ with $x \neq 0 \in \R^m$.
\end{enumerate}
Here and in the proof of Theorem~\ref{thm:random:Sobolev} (given in Section~\ref{s:proof:thm:random:Sobolev}), we will use underlined indices like $\underline\ell$ in order to refer to the indices belonging to the block with index $\ell$. For instance, in the following theorem, $\underline\ell$ is a shorthand for the indices running from $\sum_{j=1}^{\ell-1} N_j + 1$ to $\sum_{j=1}^{\ell} N_j$.

\begin{theorem}[Generalized thin $QR$ decomposition]\label{thm:generalized:QR}
	Suppose $X \in \R^{n \times D}$ has full column rank.
	Let $N_1,\ldots,N_L \in \N$ such that $\sum_{\ell=1}^L N_\ell=D$.
	Then, the generalized thin $QR$ decomposition
	\begin{equation}\label{eq:X=QR}
		X = QR
	\end{equation}
	is unique, where $Q \in \R^{n \times D}$ has orthonormal columns and $R =(R_{\underline{\ell},\underline{\ell'}}) \in \R^{D \times D}$ with $R_{\underline{\ell},\underline{\ell'}} \in \R^{N_\ell \times N_{\ell'}}$ satisfies the following properties:
	\begin{enumerate}[label=(\roman*)]
		\item $R_{\underline\ell,\underline\ell'}=0 \in \R^{N_\ell \times N_{\ell'}}$ if $\ell > \ell'$ and
		\item $R_{\underline\ell,\underline\ell}$ is positive definite.
	\end{enumerate}
\end{theorem}

\begin{proof}
	The existence part of Theorem~\ref{thm:generalized:QR} can be proven in a constructive way, for instance by adapting the classical Gram--Schmidt algorithm (see \cite{Golub2013}, Section~5.2.7) accordingly.
	The resulting algorithm is stated in Algorithm~\ref{alg:thin_QR}.
	For a positive definite matrix $A \in \R^{D \times D}$, we call a decomposition
	\begin{equation*}
		A = \Lambda \Lambda^\transpose,
	\end{equation*}
	where $\Lambda = (\Lambda_{\underline\ell,\underline\ell'}) \in \R^{D \times D}$ with blocks $\Lambda_{\underline \ell,\underline \ell'} \in \R^{N_\ell \times N_{\ell'}}$ a (generalized) Cholesky decomposition if the following properties are satisfied:
	\begin{enumerate}[label=(\roman*)]
		\item $\Lambda_{\underline\ell,\underline\ell'}=0 \in \R^{N_\ell \times N_{\ell'}}$ if $\ell < \ell'$,
		\item $\Lambda_{\underline\ell,\underline\ell}$ is positive definite.
	\end{enumerate}
	As for the usual Cholesky decomposition (which corresponds to the case $L=D$ and $N_1=\ldots=N_D=1$) the generalized Cholesky decomposition is uniquely determined given $L$ and $N_1,\ldots,N_L$.
	The factor $\Lambda$ is referred to as the \emph{Cholesky factor}.
	Note that $\Lambda=R^\transpose$, with $R$ from the generalized thin $QR$ decomposition $X=QR$, is the Cholesky factor in the generalized Cholesky factorization $P=\Lambda \Lambda^\transpose$ of the symmetric positive definite matrix $P=X^\transpose X$ (of course, for the same choice of $L$ and $N_1,\ldots,N_L$).
	The uniqueness of the generalized thin $QR$ decomposition follows from the uniqueness of the Cholesky factor $\Lambda=R^\transpose$ combined with $Q=XR^{-1}$.
\end{proof}

\begin{algorithm}
	\caption{Generalized thin $QR$ decomposition (Gram--Schmidt)}\label{alg:thin_QR}
	\begin{algorithmic}[1]
		\Require $X \in \R^{n \times D}$ of full column rank, $L \in \{1,\ldots, D\}$ and $N_1,\ldots,N_L \in \N$ satisfying $N_1+\ldots+N_L=D$
		\State Initialize (empty) matrices $Q \in \R^{n \times D}$ and $R \in \R^{D \times D}$ 
		\State $R_{\underline 1,\underline 1} \gets (X_{1:n,\underline{1}}^\transpose X_{1:n,\underline{1}})^{1/2}$ \quad (principal square root)
		\State $Q_{1:n,\underline{1}} \gets X_{1:n,\underline{1}}R_{\underline{1},\underline{1}}^{-1}$
		\For{$\ell=2,\ldots,L$}
		\State\label{line:thin_QR} $R_{\underline{1}:\underline{\ell-1},\underline{\ell}} \gets Q_{1:n,\underline{1}:\underline{\ell-1}}^\transpose X_{1:n,\underline{\ell}}$
		\State $A \gets X_{1:n,\underline{\ell}} - Q_{1:n,\underline{1}:\underline{\ell-1}}R_{\underline{1}:\underline{\ell-1},\underline{\ell}}$
		\State $R_{\underline{\ell},\underline{\ell}} \gets (A^\transpose A)^{1/2}$ \quad (principal square root)
		\State $Q_{1:n,\underline{\ell}} \gets A R_{\underline{\ell},\underline{\ell}}^{-1}$
		\EndFor
		\Ensure $Q \in \R^{n \times D}$ with orthonormal columns and $R \in \R^{D \times D}$ block upper triangular with positive definite matrices $R_{\underline{\ell},\underline{\ell}}$ on the diagonal satisfying $X=QR$
	\end{algorithmic}
\end{algorithm}

In the following proposition, we state key stochastic properties of the block upper triangular matrix $R$ in the generalized thin $QR$ decomposition~\eqref{eq:X=QR}.

\begin{proposition}\label{prop:R} Let $R$ be defined as the block upper triangular matrix $R$ with positive definite diagonal blocks in the generalized thin $QR$ decomposition $X=QR$ of the design matrix $X=(Y_j(X_i)) \in \R^{n \times D}$, where $Y_1,\ldots,Y_D$ are a complete set of spherical harmonics up to resolution level $L$ (consequently, $D=\dim \Pscr^d_L$), the numbers $N_\ell$ in the generalized thin $QR$ decomposition are chosen as $N_\ell=N_\ell^d$, and $X_1,\ldots,X_n$ are i.i.d.\ $\sim \Uc(\sphere^d)$.
Then,
\begin{enumerate}[label=(\alph*)]
	\item\label{it:prop:R:a} $\Eb[R_{\underline{\ell},\underline{\ell}}]=\lambda_\ell I_{N^d_\ell} \in \R^{N^d_\ell \times N^d_\ell}$, \quad  $\ell = 1,\ldots,L$,
	\item\label{it:prop:R:b} $\Eb[R_{\underline{\ell},\underline{\ell}'}]=0 \in \R^{N^d_\ell \times N^d_{\ell'}}$, \quad $\ell,\ell' = 1,\ldots,L$ with $\ell \neq \ell'$.
\end{enumerate}
\end{proposition}

\begin{proof}
	Let $U$ be a random rotation distributed according to the Haar measure on $\SO(d+1)$ and set $\Xtilde_i \defeq UX_i$ (here, $U$ should be interpreted as a $(d+1)\times (d+1)$-matrix and $X_i$ as an element of $\R^{d+1}$).
	Then, $\Xtilde_1,\ldots,\Xtilde_n$ are i.i.d.\ $\sim \Uc(\sphere^d)$.
	Denote with $\Xtilde = (Y_j(\Xtilde_i))$ the corresponding design matrix.
	Note that by \cite{Groemer1996}, Proposition~3.2.4, $\Xtilde$ can be written as
	\begin{equation*}
		\Xtilde = X B, 
	\end{equation*}
	where $B = (B_{\underline{\ell},\underline{\ell}'})_{\ell,\ell'=1,\ldots,L} \in \R^{D \times D}$ is a block matrix with $B_{\underline{\ell},\underline{\ell}'} = 0 \in \R^{N^d_\ell \times N^d_{\ell'}}$ if $\ell \neq \ell'$, and $B_{\underline{\ell},\underline{\ell}}$ is orthogonal.
	Let $\Xtilde = \Qtilde \Rtilde$ be the generalized thin $QR$ decomposition of $\Xtilde$, which is related to the $QR$ decomposition of $X$ via
	\begin{equation*}
		\Xtilde = XB = QRB = QBB^\transpose RB = \Qtilde \Rtilde,
	\end{equation*}
	that is, $\Qtilde=QB$ and $\Rtilde = B^\transpose R B$.
	Since $R$ and $\Rtilde$ have equal law, we have
	\begin{align*}
		\Eb[R] &= \Eb[\Rtilde]\\
		&= \Eb[B^\transpose R B]\\
		&= \Eb \Eb[B^\transpose R B | U]\\
		&= \Eb [B^\transpose\Eb[ R  | U] B]\\
		&= \Eb [B^\transpose\Eb[ R ] B].
	\end{align*}
	Combining Schur's lemma (see \cite{Serre1977}, Proposition~4 and Corollary~1, which carry over to the compact group case as explained in Section~4.3 of that reference) with the fact that the spaces $\Hscr^d_\ell$, $\ell \in \N_0$, are irreducible representations of $\SO(d+1)$ (see \cite{Dai2013}, Theorem~1.7.2) then implies both assertions \ref{it:prop:R:a} and \ref{it:prop:R:b}.
	Note that the statement of Schur's lemma carries over to the real representations considered here since the corresponding complex representations in terms of spherical harmonics are of real type.
\end{proof}

\subsubsection{A reverse Hölder inequality for spherical harmonics}

The following result, which provides even (slightly) stronger statements than what we need in the proof of Theorem~\ref{thm:random:Sobolev}, gathers some special instances of \emph{reverse Hölder inequalities} for spherical harmonics from \cite{Dai2016}.
\begin{proposition}\label{prop:reverse_Hoelder}
Let $f_\ell \in \Hscr_\ell^d$ and $X \sim \Uc(\sphere^d)$. Then,
\begin{equation*}
	\Eb[ f_\ell^4(X)] \lesssim \begin{cases}
		\lVert f_\ell \rVert_{L^2(\sphere^d)}^4\, \ell^{\frac{d-1}{2}}, & \text{ if } d \in \{2,3\},\\[2mm]
		\lVert f_\ell \rVert_{L^2(\sphere^d)}^4\, \ell^{d-2}, & \text{ if } d \geq 4.\\
	\end{cases}
\end{equation*}
\end{proposition}

\begin{proof}
	Note that $\Eb[ f_\ell^4(X)] = \int_{\sphere^d} f_\ell^4(x) \, \dd \mu (x) = \lVert f_\ell \rVert_{L^4(\sphere^d)}^4$.
	From this, the claimed assertions follow from Theorem~1.1 in \cite{Dai2016}: The cases $d=2$, $d=3$, and $d \geq 4$ follow from assertions (iv), (i), and (ii) of that theorem, respectively.
\end{proof}

\subsubsection{Further estimates}

Part~\ref{it:further:a} of the following proposition is an adaptation of \cite{Reiss2008Asymptotic}, Proposition~4.9, to our setup.

\begin{proposition}\label{prop:estimate:norm:emp:proj}
	Let $f_\ell \in \Hscr_\ell^d$ for some $\ell \leq L$ with $L$ being defined in the statement of Theorem~\ref{thm:random:Sobolev}. Then, the following assertions hold true:
	\begin{enumerate}[label=(\alph*)]
		\item\label{it:further:a} $\Eb \lVert \Pi^n_{\Pscr_{\ell-1}^d} f_\ell \rVert_n^2 \1_{\Omega_L} \lesssim \lVert f_\ell \rVert^2_{L^2(\sphere^d)} \ell^d n^{-1}$,
		\item\label{it:further:b} $\Eb \lVert \Pi^n_{\Pscr^d_{\ell-1}} f_\ell  \rVert_n^4 \lesssim \lVert f_\ell \rVert^4_{L^2(\sphere^d)} \ell^d n^{-1} + \lVert f_\ell \rVert^4_{L^2(\sphere^d)} \ell^{2d-2} n^{-2}$.
	\end{enumerate}
\end{proposition}

\begin{proof}
	We have
	\begin{align*}
		\Eb \lVert \Pi^n_{\Pscr_{\ell-1}^d}f_{\ell} \rVert_n^2 \1_{\Omega_L} &= \Eb \left[ \sup_{p \in \Pscr^d_{\ell-1}} \, \frac{\lvert \langle f_\ell,p \rangle_n \rvert^2}{\lVert p \rVert_n^2} \, \1_{\Omega_L} \right]\\
		&= \Eb \left[  \sup_{\substack{p \in \Pscr^d_{\ell-1}\\\lVert p\rVert_{L^2(\sphere^d)}=1}} \, \lvert \langle f_{\ell},p \rangle_n \rvert^2 \, \cdot \sup_{p \in \Pscr^d_{\ell-1}} \frac{\lVert p \rVert_{L^2(\sphere^d)}^2}{\lVert p \rVert_n^2} \, \1_{\Omega_L} \right]\\
		&\lesssim \Eb \left[ \sup_{\substack{\lVert c \rVert = 1}} \left| \frac{1}{n} \sum_{i=1}^{n} f_{\ell}(X_i) \sum_{j = 1}^{\dim \Pscr^d_{\ell-1}} c_{j} Y_{j}(X_i) \right|^2 \right]\\
		&\leq \Eb \left[ \sum_{j=1}^{\dim \Pscr^d_{\ell-1}} \left| \frac{1}{n} \sum_{i=1}^{n} f_{\ell}(X_i) Y_{j}(X_i) \right|^2 \right]\\
		&= \frac{1}{n}\, \Eb \left[  f_\ell^2(X_1) \sum_{j=1}^{\dim \Pscr^d_{\ell-1}} Y_{j}^2(X_1) \right]\\
		&= \lVert f_\ell \rVert^2_{L^2(\sphere^d)} \cdot \frac{\dim \Pscr^d_{\ell-1}}{n}\\
		&\asymp \lVert f_\ell \rVert^2_{L^2(\sphere^d)}  \ell^d n^{-1},
	\end{align*}
	which proves \ref{it:further:a}.
	For the proof of \ref{it:further:b}, we use the decomposition
	\begin{equation*}
		\Eb \lVert \Pi^n_{\Pscr^d_{\ell-1}} f_\ell  \rVert_n^4 = \Eb\lVert \Pi^n_{\Pscr^d_{\ell-1}} f_\ell  \rVert_n^4 \1_{\Omega_L} + \Eb \lVert \Pi^n_{\Pscr^d_{\ell-1}} f_\ell  \rVert_n^4 \1_{\Omega^\complement_L}.
	\end{equation*}
	For the first term on the right-hand side, we have by \ref{it:further:a} and the defining property of the event $\Omega_L$ that
	\begin{align*}
		\Eb\lVert \Pi^n_{\Pscr^d_{\ell-1}} f_\ell  \rVert_n^4 \1_{\Omega_L} &\leq \Eb\lVert \Pi^n_{\Pscr^d_{\ell-1}} f_\ell  \rVert_n^2 \lVert f_\ell \rVert_n^2 \1_{\Omega_L}\\
		&\leq \frac{3}{2} \lVert f_\ell \rVert^2_{L^2(\sphere^d)} \Eb\lVert \Pi^n_{\Pscr^d_{\ell-1}} f_\ell  \rVert_n^2 \1_{\Omega_L}\\
		&\lesssim \lVert f_\ell \rVert^4_{L^2(\sphere^d)} \ell^d n^{-1}.
	\end{align*}
	The second term on the right-hand side is bounded by
	\begin{align*}
		\Eb \lVert \Pi^n_{\Pscr^d_{\ell-1}} f_\ell  \rVert_n^4 \1_{\Omega^\complement_L} &\leq \Eb \lVert f_\ell  \rVert_n^4 \1_{\Omega^\complement_L}\\
		&= \Eb \left[ \left( \frac{1}{n} \sum_{i=1}^{n} f_\ell^2(X_i) \right)^2  \1_{\Omega^\complement_L} \right]\\
		&\leq \lVert f_\ell \rVert^4_{L^2(\sphere^d)} (N^d_\ell)^2 \Pb(\Omega_L^\complement)\\
		&\lesssim \lVert f_\ell \rVert^4_{L^2(\sphere^d)} \ell^{2d-2} n^{-2}.
	\end{align*}
	Combining the two obtained bounds implies \ref{it:further:b}.
\end{proof}

\printbibliography

\end{document}